\documentclass[a4paper,11pt]{article}

\usepackage{amsmath}
\usepackage{amssymb}
\usepackage{mathtools}
\usepackage{amsthm}
\usepackage{thmtools}
\usepackage{thm-restate}
\usepackage{graphicx}
\usepackage{amscd}
\usepackage{epic, eepic}
\usepackage{url}
\usepackage{color}
\usepackage[utf8]{inputenc} 
\usepackage{comment}
\usepackage{enumerate}
\usepackage{graphicx}
\usepackage{epstopdf}
\usepackage{enumitem}
\usepackage{caption}
\usepackage{subcaption}
\usepackage[top=20 mm, bottom=22 mm, left=20 mm, right= 20 mm]{geometry}
\usepackage[bookmarks=false,hidelinks]{hyperref}

\usepackage{tikz}
\usetikzlibrary{math,calc,intersections}
\usetikzlibrary{positioning,arrows,shapes,decorations.markings,decorations.pathreplacing,matrix,patterns}
\tikzstyle{vertex}=[circle,draw=black,fill=black,inner sep=0,minimum size=5pt,text=white,font=\footnotesize]

\makeatletter
\renewenvironment{proof}[1][\proofname] {\par\pushQED{\qed}\normalfont\topsep6\p@\@plus6\p@\relax\trivlist\item[\hskip\labelsep\bfseries#1\@addpunct{.}]\ignorespaces}{\popQED\endtrivlist\@endpefalse}
\makeatother

\newtheorem{theorem}{\bf Theorem}[section]
\newtheorem{lemma}[theorem]{\bf Lemma}
\newtheorem{claim}[theorem]{\bf Claim}
\newtheorem{corollary}[theorem]{\bf Corollary}
\newtheorem{proposition}[theorem]{\bf Proposition}

\newtheorem*{theorem*}{\bf Theorem}

\theoremstyle{definition}

\def\eps{\varepsilon}

\def\cC{\mathcal{C}}

\def\cG{\mathcal{G}}
\def\cH{\mathcal{H}}
\def\cI{\mathcal{I}}

\def\cP{\mathcal{P}}
\def\ex{\mathrm{ex}}

\def\bE{\mathbb{E}}

\def\fB{\mathfrak{B}}

\newcommand{\hide}[1]{}
\newcommand{\ind}{\text{-ind}}

\title{\vspace{-0.9cm} K\H{o}v\'ari-S\'os-Tur\'an theorem for hereditary families}
\author{Zach Hunter\thanks{ETH Zurich, e-mail: \textbf{\{zach.hunter, aleksa.milojevic, benjamin.sudakov\}@math.ethz.ch}. Research supported in part by SNSF grant 200021-228014.}, Aleksa Milojevi\'c\footnotemark[1], Benny Sudakov\footnotemark[1], Istv\'an Tomon\thanks{Ume\r{a} University, \emph{e-mail}: \textbf{istvan.tomon@umu.se}. Research supported in part by the Swedish Research Council grant VR 2023-03375.}}

\date{}

\begin{document}

\maketitle

\begin{abstract}
    The celebrated K\H{o}v\'ari-S\'os-Tur\'an theorem states that any $n$-vertex graph containing no copy of the complete bipartite graph $K_{s,s}$ has at most $O_s(n^{2-1/s})$ edges. In the past two decades, motivated by the applications in discrete geometry and structural graph theory, a number of results demonstrated that this bound can be greatly improved if the graph satisfies certain structural restrictions. We propose the systematic study of this phenomenon, and state the conjecture that if $H$ is a bipartite graph, then an induced $H$-free and $K_{s,s}$-free graph cannot have much more edges than an $H$-free graph. We provide evidence for this conjecture by considering trees, cycles, the cube graph, and bipartite graphs with degrees bounded by $k$ on one side, obtaining in all the cases similar bounds as in the non-induced setting. Our results also have applications to the Erd\H{o}s-Hajnal conjecture, the problem of finding induced $C_4$-free subgraphs with large degree and bounding the average degree of $K_{s, s}$-free graphs which do not contain induced subdivisions of a fixed graph.
\end{abstract}

\section{Introduction}

The goal of this paper is to combine two classical areas of graph theory: Tur\'an problems and the study of graphs with forbidden induced subgraphs. The \emph{extremal number} or \emph{Tur\'an number} of a graph $H$ is the maximum number of edges in an $n$-vertex graph containing no copy of $H$ as a subgraph, and it  is denoted by  $\ex(n,H)$. The study of extremal numbers goes back more than a 100 years to Mantel \cite{Mantel}, who determined the extremal number of the triangle. This was extended by Tur\'an \cite{T41}, who found the extremal number of every clique. By the Erd\H{o}s-Stone-Simonovits theorem \cite{ES66,ES46}, we know the extremal number of every non-bipartite graph $H$ up to lower order terms.

The extremal numbers of bipartite graphs are much more mysterious, with a plethora of results addressing specific instances, and with just as many open problems. One of the most notorious problems is to determine the extremal number of $K_{s,t}$, i.e. the complete bipartite graph with vertex classes of size $s$ and $t$. The celebrated K\H{o}v\'ari-S\'os-Tur\'an  theorem \cite{KST54} states that if $s\leq t$, then $\ex(n,K_{s,t})=O_t(n^{2-1/s})$. This is only known to be tight if $s\in \{2,3\}$, or $t$ is sufficiently large with respect to $s$ \cite{ARS99,Bukh}. On the other hand, the random deletion method shows that $\ex(n,K_{s,s})=\Omega_s(n^{2-2/(s+1)})$.

Another classical topic of graph theory is the study of graphs avoiding a fixed graph $H$ as an induced subgraph. One such problem closely related to the topic of this paper is the Gy\'arf\'as-Sumner conjecture \cite{Gyarfas,Sumner}. A family of graphs $\mathcal{G}$ is \emph{$\chi$-bounded} if there exists a function $f$ such that $\chi(G)\leq f(\omega(G))$ for every $G\in \mathcal{G}$, where $\chi(G)$ denotes the chromatic number, and $\omega(G)$ the clique number. In this case, we say that $f$ is a \emph{$\chi$-bounding function} for $\mathcal{G}$. The Gy\'arf\'as-Sumner conjecture states that if $T$ is a tree, then the family of graphs avoiding $T$ as an induced subgraph is $\chi$-bounded. An additional classical problem of interest about graphs avoiding a fixed induced subgraph is the Erd\H{o}s-Hajnal conjecture \cite{EH}. This conjecture states that if $H$ is a graph, then any induced $H$-free $n$-vertex graph contains either a clique or an independent set of size at least $n^{c}$ for some $c=c(H)>0$.

Here, we consider the problem of finding the maximum number of edges in a $K_{s,s}$-free graph, assuming the host graph satisfies certain further structural restrictions. To this end, given a family of graphs $\mathcal{G}$, let $\ex_{\mathcal{G}}(n,s)$ denote the maximum number of edges of an $n$-vertex member of $\mathcal{G}$ which contains no copy of $K_{s,s}$.  In the past two decades, the function $\ex_{\mathcal{G}}(n,s)$ has been extensively studied for various natural (typically hereditary) families $\mathcal{G}$. In each case, it has been observed that the trivial bound $\ex_{\mathcal{G}}(n,s)\leq \ex(n,K_{s,s})=O_s(n^{2-1/s})$ can be significantly improved. The study of these problems developed both in structural graph theory and combinatorial geometry, seemingly independently from each other. One goal of our manuscript is to provide a systematic study of  $\ex_{\mathcal{G}}(n,s)$, and to unite these two areas. Let us survey the known results.

\medskip

\noindent
\textbf{Structural graph theory.} Given a graph $H$, or a family of graphs $\mathcal{H}$, we are interested in the family $\mathcal{G}$ defined as the family of all graphs containing no induced copy of a member of $\mathcal{H}$. In this case, let us write $\ex^{*}(n,\mathcal{H},s):=\ex_{\mathcal{G}}(n,s)$, and simply $\ex^{*}(n,H,s)$ if $\mathcal{H}=\{H\}$. It was proved by K\"uhn and Osthus \cite{KO04} that if $\mathcal{H}$ is the family of  all subdivisions of a fixed graph $H$, then $\ex^*(n,\mathcal{H},s)=O_s(n)$. The constant hidden in the $O_s(.)$ notation was recently improved by Du, Gir\~ao, Hunter, McCarty and Scott \cite{DGHMS23}. Bonamy et al. \cite{BBPRTW} showed that $\ex^{*}(n,P_t,s)=s^{O_t(1)}n$ and $\ex^{*}(n,\mathcal{C}_{\geq t},s)=s^{O_t(1)}n$, where $P_t$ is the path of length $t$ and $\mathcal{C}_{\geq t}$ is the family of cycles of length at least $t$. A common strengthening of the previous results is proved independently by  Gir\~ao and Hunter \cite{GH23} and Bourneuf, Buci\'c, Cook and Davies \cite{BBCD23}: if $\mathcal{H}$ is the family of subdivisions of a graph $H$, then $\ex^*(n,\mathcal{H},s)\leq s^{O_{H}(1)}n$. As another generalization of the result on paths, Scott, Seymour and Spirkl \cite{SSS}, improving a previous unpublished result of R\"odl, proved that for every tree $T$, one has  $\ex^{*}(n,T,s)=s^{O_T(1)}n$, where the exponent of $s$  has the order of magnitude $|T|^{\Omega(|T|)}$.

These results are partially motivated by the above described Gy\'arf\'as-Sumner conjecture, which says that the family of graphs avoiding an induced copy of a fixed tree is $\chi$-bounded.
Gy\'arf\'as \cite{Gyarfas} proved that the conjecture is true for every path, while Kierstead and Penrice \cite{KP94} proved it if $T$ is a tree of radius two. In general, Scott \cite{Scott97} showed that if $T$ is a tree, then the family of graphs avoiding all induced subdivisions of $T$ is $\chi$-bounded. Scott \cite{Scott97} also proposed the conjecture that this holds if we forbid all subdivisions of a given graph $H$, however, this is disproved by Pawlik et al. \cite{PaKK14}. The \emph{polynomial Gy\'arf\'as-Sumner conjecture} states that one can also find a polynomial $\chi$-bounding function for families avoiding induced copies of a tree $T$. This conjecture is open even for paths with at least 5 vertices.

Observe that forbidding $K_{s,s}$ can be thought of as a relaxation of the condition that $\omega(G)\leq s$, which is equivalent to forbidding the complete graph $K_{s+1}$. Thus, considering $K_{s, s}$-free graphs with no induced copies of $H$ is an natural intermediate step in proving the Gy\'arf\'as-Sumner conjecture for some specific trees $H$, see e.g. \cite{KP94, KZ04}. If $\mathcal{G}$ is a hereditary family of graphs (i.e. a family of graphs closed under taking induced subgraphs), then the inequality $\ex_{\mathcal{G}}(n,s)\leq c(s) n$ implies that the $K_{s,s}$-free members of $\mathcal{G}$ are $2c(s)$-degenerate (i.e. every subgraph has a vertex of degree at most $2c(s)$). This implies that $K_{s,s}$-free members of $\mathcal{G}$ have chromatic number at most $2c(s)+1$. In particular, the results described above show that the family of graphs containing no $K_{s,s}$ and no induced subdivision of a given graph $H$ have chromatic number at most $s^{O_H(1)}$. 

 Gir\~ao and Hunter \cite{GH23} and Bourneuf, Buci\'c, Cook and Davies \cite{BBCD23} also consider the family of graphs avoiding an induced copy of a bipartite graph $H$. Namely, they show that for every $H$, there exists  $\eps_H>0$ such that $\ex^*(n, H, s)\leq O_{s, H}(n^{2-\eps_H})$. Bourneuf, Buci\'c, Cook and Davies \cite{BBCD23} find $\eps_H$ that is exponential in the size of $H$, while Gir\~ao and Hunter \cite{GH23} show that $\eps_H$ can be taken to be $\frac{1}{100\Delta(H)}$, where $\Delta(H)$ is the maximum degree of $H$. 

\medskip

\noindent
\textbf{Combinatorial geometry.} In this area we are interested in the following types of graphs. The \emph{intersection graph} of a family $\mathcal{F}$ is the graph on vertex set $\mathcal{F}$, where two sets are joined by an edge if they have a nonempty intersection. The \emph{bipartite intersection graph} of two families $\mathcal{A}$ and $\mathcal{B}$ is the bipartite graph with vertex classes $\mathcal{A}$ and $\mathcal{B}$, with edges joining $A\in \mathcal{A}$ and $B\in\mathcal{B}$ if $A\cap B\neq \emptyset$. The \emph{incidence graph} of a set $X$ and a family of sets $\mathcal{F}$ is the bipartite graph with vertex classes $X$ and $\mathcal{F}$, where $x\in X$ is joined to $A\in\mathcal{F}$ if $x\in A$.

A \emph{curve} in the plane is the image of a continuous function $\phi:[0,1]\rightarrow \mathbb{R}^2$, and a \emph{string graph} is an intersection graph of a collection of curves in the plane. These graphs are extensively studied both in computational and combinatorial geometry. It is well known that string graphs avoid induced proper subdivisions of $K_5$, where a subdivision is proper if every edge is subdivided at least once. This immediately implies that $n$-vertex $K_{s,s}$-free string graphs have at most $O_s(n)$ edges by the results mentioned above \cite{BBCD23,DGHMS23,GH23,KO04}. The optimal bound $O(s(\log s)n)$ is obtained by Lee \cite{Lee}, building on ideas of Fox and Pach \cite{FP08}.

Chan and Har-Peled \cite{CH23} considered incidence graphs of $n$ points and $n$ pseudo-discs in the plane. A family of simple closed Jordan regions is a family of \emph{pseudo-discs} if the boundary of any two intersect in at most two points. They proved that if such an incidence graph is $K_{s,s}$-free, then it has $O_s(n\log\log n)$ edges. This was strengthened by Keller and Smorodinsky \cite{KS23} who proved that if $\mathcal{A}$ and $\mathcal{B}$ are families of $n$ pseudo-disks, then the bipartite intersection graph of $\mathcal{A}$ and $\mathcal{B}$ has at most $O(s^6 n)$ edges, assuming it is $K_{s,s}$-free. We highlight that by a result Keszegh \cite{Keszegh}, such graphs contain no induced proper subdivision of a non-planar graph, in particular no induced proper subdivision of $K_5$. Hence, one immediately gets the bound $s^{O(1)}n$ by \cite{BBCD23,GH23}. The study of the family $\mathcal{G}_{\text{box}}$ of incidence graphs of points and axis-parallel boxes in $\mathbb{R}^d$ was initiated by Basit, Chernikov, Starchenko, Tao and Tran \cite{BCSTT}. Chan and Har-Peled \cite{CH23} determined the optimal bound $\ex_{\mathcal{G}_{\text{box}}}(n,s)=O_s(n(\frac{\log n}{\log\log n})^{d-1})$. Furthermore, Tomon and Zakharov \cite{TZ} study the closely related family of intersection graphs of axis-parallel boxes in $\mathbb{R}^d$, see also \cite{KS23}. 

As a high-dimensional generalization of the Szemer\'edi-Trotter theorem \cite{SzT}, Chazelle \cite{Ch93} proposed to study $K_{s,s}$-free incidence graphs of points and hyperplanes in $\mathbb{R}^d$. In this context, forbidding $K_{s, s}$ in the incidence graph is a natural nondegeneracy condition, since otherwise all points may lie on a line which is contained in all hyperplanes, in which case no nontrivial upper bound on the number of incidences can be given. For $d\geq 3$, the currently best known upper bound on the number of edges in a $K_{s,s}$-free incidence graph of $n$ points and $n$ hyperplanes is $O_{d,s}(n^{2-2/(d+1)})$, proved by Apfelbaum and Sharir \cite{AS07}, while the best known lower bound is due to Sudakov and Tomon \cite{ST23}. In \cite{MST}, Milojevi\'c, Sudakov and Tomon show that over any field $\mathbb{F}$, a $K_{s,s}$-free incidence graph of $n$ points and $n$ hyperplanes has at most  $O_{d,s}(n^{2-\frac{1}{\lceil (d+1)/2\rceil}})$ edges, and this bound is the best possible. The proof of this result proceeds by considering $\ex^*(n,\mathcal{H},s)$ for a simple finite family of bipartite graphs $\mathcal{H}$ avoided by point-hyperplane incidence graphs. In another direction, Fox, Pach, Scheffer, Suk and Zahl \cite{FPSSZ} proved that the bound of Apfelbaum and Sharir \cite{AS07} also holds (up to an $o(1)$ error term in the exponent) for $K_{s,s}$-free semialgebraic graphs in dimension $d$ as well (we refer the interested reader to \cite{FPSSZ} for precise definitions). One of their key tools is a bound on the number of edges in a $K_{s,s}$-free graph of VC-dimension at most $d$. Here, a graph has VC-dimension at most $d$ if it contains no set $A$ of $d+1$ vertices such that for every $X\subset A$ there is a vertex joined to every vertex in $X$, but no vertex in $A\setminus X$. In \cite{FPSSZ}, it is proved that a $K_{s,s}$-free graph of VC-dimension at most $d$ has at most $O_{d,s}(n^{2-1/d})$ edges, see also \cite{KS23} for an alternative proof. For $d\geq 3$, this was improved to $o(n^{2-1/d})$ by Janzer and Pohoata~\cite{JP}. For a generalization of this result to an even wider class of graphs by Axenovich and Zimmermann, see \cite{AZ}. 

\medskip

\noindent
\textbf{Further related work.} Loh, Tait, Timmons and Zhou \cite{LTTZ18} have defined the notion of \emph{induced Tur\'an number} for graphs $H$ and $F$. This is the maximum number of edges in an $n$ vertex graph $G$ containing no induced copy of $F$ and no (not necessarily induced) copy of $H$, and is denoted by $\ex(n, \{H, F\ind\})$. In \cite{LTTZ18}, they considered  $H=K_r$ and $F=K_{s, t}$, obtaining the bound $e(G)\leq O_{r, s, t}(n^{2-1/s})$. They have also considered the cases when $F=K_{2, t}$ and $H$ is either an odd cycle or a complete graph. Their work was later extended by Illingworth (\cite{Illingworth21}, \cite{Illingworth21a}) and Ergemlidze, Gy\H{o}ri and Methuku \cite{EGM19}. These results are quite different from ours, since we consider hereditary families in which induced copies of an arbitrary bipartite graph $F$ are forbidden, instead of just considering cases when $F$ is a complete bipartite graph.

\subsection{Main results}
In this section, we present our results. Recall that $\ex^{*}(n,H,s)$ denotes the maximum number of edges in an $n$-vertex graph containing no $K_{s,s}$ and no induced copy of $H$. We consider $\ex^{*}(n,H,s)$ for various bipartite graphs $H$. Note that the case when $H$ is not bipartite is not very interesting: as every graph contains a bipartite subgraph with at least half of the edges, we have $\ex^{*}(n,H,s)=\Theta(\ex(n,K_{s,s}))$. In the case where 
$H$ is bipartite, we believe that for sufficiently large $s$, $\ex^*(n,H,s)$ is close to the extremal number of $H$. In particular, we propose the following conjecture.

\begin{restatable}{conjecture}{mainconjecture}\label{conj:mainconjecture}
For every connected bipartite graph $H$, 
$$\ex^{*}(n,H,s)\leq C_H(s)\cdot \ex(n,H)$$
for some $C_H(s)$ depending only on $H$ and $s$.    
\end{restatable}

 This conjecture is so far consistent with all known results. We provide further evidence by considering several families of bipartite graphs for which the extremal numbers are extensively studied. First, we consider bipartite graphs $H=(A, B; E)$ in which all vertices in $B$ have degree at most $k$. Such bipartite graphs are of interest due to a celebrated result of F\"uredi \cite{Furedi} and Alon, Krivelevich, Sudakov \cite{AKS} which shows that $\ex(n,H)=O_{H}(n^{2-1/k})$. In general, this bound is also the best possible by taking $H=K_{k,t}$ for sufficiently large $t$. We show that a similar bound holds for $\ex^{*}(n,H,s)$ as well.

\begin{restatable}{theorem}{boundeddeg}\label{thm:bounded degree}
Let $H=(A, B; E)$ be a bipartite graph such that every vertex in $B$ has degree at most~$k$. Then, \[\ex^*(n, H, s)\leq (C_H s)^{4|V(H)|+10} n^{2-1/k},\] with $C_H=4|A||B|$.
\end{restatable}

\noindent
This result improves the above mentioned bounds of Gir\~ao and Hunter \cite{GH23} and Bourneuf, Buci\'c, Cook and Davies \cite{BBCD23}. The exponent $2-1/k$
is optimal, which can be seen by 
considering $H=K_{k,t}$ for sufficiently large $t$.

Next, we consider trees. It is an easy exercise to show that if $T$ is a tree, then $\ex(n,T)\leq |T|n$. As proved by Scott, Seymour and Spirkl \cite{SSS}, $\ex^*(n,T,s)$ also grows linearly as a function of $n$, and furthermore as a polynomial function of $s$. The exponent of $s$ in their result is super-exponential  in $|T|$, it is of the order $|T|^{\Omega(|T|)}$. One of our main results improves this to linear, i.e., $\ex^{*}(n,T,s)=s^{O(|T|)}n$. This is optimal up to the constant term hidden by $O(.)$ for every tree $T$.

\begin{theorem}\label{thm:trees}
There exists an absolute constant $C$ such that for every tree $T$ on $t$ vertices, and for every $s$ sufficiently large with respect to $t$, we have $$\ex^*(n, T, s)\leq s^{Ct} n.$$
\end{theorem}

Let us argue why this theorem is optimal by showing that $\ex^*(n, T, s)\geq s^{\Omega(t)}n $ for $s\gg t^2$. Consider the random graph on $N=s^{t/10}$ vertices, in which each edge is included with probability $1-s^{-1/2}$. It is a standard union bound argument to show that this graph contains no $K_{s,s}$, no independent set of size at least $t/2$ (which ensures that it is also induced $T$-free), and has at least $N^2/4$ edges with high probability. Taking $n/N$ disjoint copies of such a graph proves our lower bound. 
 
Next, we consider the cycle of length $2k$, denoted by $C_{2k}$. The classical result of Bondy and Simonovits \cite{BS74} states that $\ex(n,C_{2k})=O_k(n^{1+1/k})$, and this bound is known to be tight for $k\in \{2,3,5\}$ \cite{LUW}. We achieve a similar upper bound for $\ex^{*}(n,C_{2k},s)$ as well.

\begin{restatable}{theorem}{cycles}\label{thm:cycles}
Let $k,s\geq 2$ be integers, then there exists $C_{s,k}$ such that $$\ex^*(n, C_{2k}, s)\leq C_{s, k} n^{1+1/k}.$$ 
\end{restatable}

Finally, we consider $Q_8$, the graph of the cube. That is, the vertices of $Q_{8}$ are $\{0,1\}^3$, and two vertices are connected by an edge if they differ in exactly one coordinate. Determining the order of $\ex(n,Q_8)$ is an old open problem of Erd\H{o}s \cite{E64}, and the best known upper bound $\ex(n,Q_8)=O(n^{8/5})$ is due Erd\H{o}s and Simonovits \cite{ES}. We show the same upper bound for $\ex(n,Q_8,s)$ as well.

\begin{proposition}\label{prop:cube}
For every integer $s\geq 2$, there exists $C_s$ such that $$\ex^*(n, Q_8, s)\leq C_{s} n^{8/5}.$$
\end{proposition}

\subsection{Applications}

Our results and techniques have several additional applications which we discuss in this subsection.

K\"uhn and Osthus \cite{KO} proved that for every $k$, every graph of sufficiently large average degree contains a $C_4$-free subgraph of average degree at least $k$ (see also \cite{MPS} for improved bounds). McCarty \cite{McCarty} established a natural induced variant of this result: for every $k$ and sufficiently large $s$, there exists a smallest number $g(k,s)$ such that if a $K_{s,s}$-free graph $G$ has average degree at least $g(k,s)$, then $G$ contains a $C_4$-free \emph{induced} subgraph of average degree at least $k$. Quantitative bounds are first proved by Du, Girão, Hunter, McCarty and Scott \cite{DGHMS23}, who show that $g(k,s)\le k^{O(s^3)}$, while Girão and Hunter \cite{GH23} prove that $g(k,s)\le  s^{O(k^4)}$. In \cite{GH23}, the lower bound $g(k,s) \ge s^{\Omega(k^2)}$ is also established for $s$ sufficiently large with respect to $k$. Our results can be used to show that this lower bound is tight and $g(k, s)\leq s^{O(k^2)}$.

\begin{restatable}{theorem}{Cfourfree}\label{thm:induced C_4-free}
Let $k\ge 1$ be an integer and let $s$ be sufficiently large compared to $k$. Then every $K_{s, s}$-free graph $G$ with average degree at least $s^{10^3 k^2}$ contains an induced subgraph with no  $C_4$ and average degree at least $k$.
\end{restatable}

Note that quantitative bounds for $g(k,s)$ are quite useful, since there are interesting hereditary families $\cG$ which do not contain $C_4$-free graphs of large average degree $k$. For such families, we know that the average degree of any $K_{s, s}$-free graph $G\in \cG$ is at most $g(k, s)$ and therefore $\ex_{\cG}(n, s)\leq g(k, s)n$.
For example, the family $\cG$ of graphs which do not contain an induced subdivision of a fixed graph $H$ has this property, by a result of K\"uhn and Osthus \cite{KO04}. Moreover, the proof of Theorem~\ref{thm:induced C_4-free} can be used to show a tight bound on the Tur\'an number of induced subdivisions in $K_{s, s}$-free graphs, improving the bounds from \cite{KO04, DGHMS23, GH23, BBCD23}. 

\begin{theorem}\label{thm:induced subdivisions turan}
Let $H$ be a fixed bipartite graph, $s$ sufficiently large integer and let $G$ be a $K_{s, s}$-free graph which does not contain an induced subdivision of $H$. Then, the average degree of $G$ is at most $s^{O(|V(H)|)}$.
\end{theorem}

\noindent
The fact that this bound is tight up to the constant in the exponent of $s$ follows from the same random graph construction discussed after the statement of Theorem~\ref{thm:trees}.

A well known lopsided weakening of the Erd\H{o}s-Hajnal conjecture, proved by Fox and Sudakov \cite{FS09}, says that every induced $H$-free graph on $n$ vertices contains either a complete bipartite graph with vertex classes of size $n^{c}$, or an independent set of size at least $n^{c}$, for some $c=c(H)>0$. Their proof implies that one can take $c(H)=\Omega(1/|V(H)|^3)$. We improve this in case $H$ is bipartite to $c(H)=\Omega(1/|V(H)|)$, which is optimal as can be shown by considering appropriate random graphs. We also remark that a recent result of Nguyen, Scott and Seymour \cite{NSS} proves the Erd\H{o}s-Hajnal conjecture in the case one forbids \emph{bi-induced copies} of a bipartite graph $H$, i.e.\ we forbid every induced copy of those graphs we get by possibly adding edges to the two vertex classes of $H$.

\begin{restatable}{theorem}{ErdosHajnal}\label{thm:half Erdos Hajnal}
Let $H$ be a bipartite graph on $h$ vertices. Then, for every sufficiently large $n$, every graph $G$ on $n$ vertices with no induced copy of $H$ contains either an independent set of size $n^{\Omega(1/h)}$ or a complete bipartite graph with parts of size at least $n^{\Omega(1/h)}$.
\end{restatable}

\section{Bipartite graphs with bounded degree on one side}\label{sec:bounded}

The main goal of this section is to prove Theorem~\ref{thm:bounded degree}, which we now restate in a slightly extended form. 

\begin{theorem}\label{thm:bounded deg variant}
Let $H=(A,B;E)$ be a bipartite graph such that the degree of every vertex in $B$ is at most $k$, let  $C_H=4|A||B|$ and let $G$ be a $K_{s,s}$-free graph on $n$ vertices.
\begin{itemize}
    \item[(i)] If $e(G)\geq (C_Hs)^{4|V(H)|+10}n^{2-1/k}$, then $G$ contains an induced copy of $H$.

    \item[(ii)] If  $n\geq (C_Hs)^{8|V(H)|+20}$ and $e(G)\geq n^{2-1/4k}$, then $G$ contains an induced copy of $H$.
\end{itemize} 
\end{theorem}

Although the second statement of Theorem~\ref{thm:bounded deg variant} is weaker than the first one for large $n$, it will be used in the proof of some of our other results. We develop new ideas and machinery in order to guarantee that the exponent of $s$ is linear in $|V(H)|$. Indeed, there are shorter proofs available that achieve an exponent of $s$ that is quadratic in $|V(H)|$ (see \cite{GH23} for similar arguments).  However, the methods used for obtaining this linear dependence are crucial for getting optimal results in our applications.

The proof of Theorem \ref{thm:bounded deg variant} has three main steps, which we outline now. Throughout this section, we fix a host graph $G$ with $n$ vertices and at least $C n^{2-1/k}$ edges which does not contain $K_{s, s}$ as a subgraph, and we fix the bipartite graph $H=(A,B;E)$. Furthermore, we denote the size of $A$ and $B$ by $a$ and $b$, and write  $|V(H)|=h=a+b$. 

\begin{description}
\item[Step 1.] Find a large set $X\subset V(G)$ in which every $k$-tuple of vertices has a large common neighbourhood. This step is a standard application of the dependent random choice technique \cite{FS}, which is a technique used in the proof that the usual Tur\'an number of $H$ is $O(n^{2-1/k})$.

\item[Step 2.] Construct \emph{rich} independent sets in $X$. Here, we call a set of vertices $S\subset V(G)$ \textit{rich} if for all subsets $T\subseteq S$ of size at most $k$, there are at least $(4bs)^{b}$ vertices $v\in V(G)\backslash S$ which are adjacent to all vertices of $T$ and non-adjacent to all vertices of $S\backslash T$. This step is the most  technically involved part of our proof.

\item[Step 3.] Show that if $I$ is a rich independent set of size $a$, then there is a set $U\subset V(G)$ of size $b$ such that $U\cup I$ induces a copy of $H$.
\end{description}

After presenting the proof of Theorem~\ref{thm:bounded deg variant}, we give a simple construction that shows why the upper bounds we obtain must depend polynomially on $s$.  

\subsection{Embedding the vertices of degree at most $k$}

We start with step 3, that is, we show how to find an induced copy of $H$ given a rich independent set. We begin this section by presenting a simple statement that will be used repeatedly in our proofs.

\begin{claim}\label{claim:simple}
Let $G$ be a graph which does not contain $K_{s,s}$. Then for every set $W$ of size at least $2s$, there are at most $s$ vertices $v\in V(G)$ for which $|W\backslash N(v)|\leq |W|/2s$.
\end{claim}
\begin{proof}
Suppose there are $s$ vertices $v_1, \dots, v_s\in V(G)$ such that $|N(v_i)\cap W|\geq \left(1-\frac{1}{2s}\right)|W|$. Then, the common neighbourhood of $v_1, \dots, v_s$ in $W$ has size at least $|W|-s\cdot \frac{1}{2s}|W|= |W|/2\geq s$. This contradicts that $G$ is $K_{s, s}$-free.
\end{proof}

\noindent
 Recall that a set of vertices $S\subset V(G)$ is \textit{rich} if for each $T\subseteq S$, $|T|\leq k$, we have $$|\{v\in V(G)\backslash S:N(v)\cap S=T\}|\geq (4bs)^{b}.$$

\begin{restatable}{lemma}{embeddingB}\label{lemma:greedy} 
Let $G$ be a graph not containing $K_{s,s}$. If $G$ contains a rich independent set of size $a$, then $G$ contains $H$ as an induced subgraph.
\end{restatable}

\begin{proof}

Let $I$ be a rich independent set of size $|A|$. We embed the vertices of $A$ into $I$ in an arbitrary manner and show that we can extend this embedding to obtain an induced copy of $H$. We index the vertices of $B$ by $b_1, \dots, b_{|B|}$, and we denote the embedding of vertices of $A$ into $I$ by $\phi$.

We embed the vertices $b_1, b_2, \dots, b_{|B|}$ inductively. Assume that we already defined the images $\phi(b_1), \dots, \phi(b_i)$. A vertex  $v\in V(G)\backslash(\{\phi(b_1), \dots, \phi(b_i)\}\cup I)$ is a \textit{candidate} for $b_j$ if it is not adjacent to $\phi(b_1), \dots, \phi(b_i)$, and  satisfies $N(v)\cap I=\phi(N_H(b_j))$. The first condition serves to ensure that the embedding of $B$ forms an independent set, while the second one guarantees that the edges between $\phi(A)$ and $\phi(B)$ correspond precisely to the edges of $H$. We show by induction on $i=0, 1, \dots, |B|$ that one can embed the vertices $b_1, \dots, b_i$ such that for each $j>i$, there are at least $(4|B|s)^{|B|-i}$ candidates for $b_j$. When $i=0$, there are at least $(4|B|s)^{|B|}$ candidates for each $b_j$, since $I$ is a rich independent set and $b_j$ has at most $k$ neighbours. Having defined $\phi(b_1), \dots, \phi(b_{i})$, our goal is to define $\phi(b_{i+1})$ such that the sets of candidates have not decreased too much. More precisely, if we denote the set of candidates for $b_j$ by $U_j$, our goal is to find a vertex $v\in U_{i+1}$ for which $|U_j\backslash N(v)|\geq |U_j|/2s$ for all $j\geq i+2$. If we find such a vertex $v$, then we take $\phi(b_{i+1}):=v$.

The existence of such a vertex $v$ is a simple consequence of Claim~\ref{claim:simple}. Namely, by the induction hypothesis, for $i\leq |B|-1$ we have $|U_j|\geq (4|B|s)^{|B|-i}\geq 2s$ for all $j>i$ and therefore for each $j>i$ there are at most $s$ vertices $v$ in $U_{i+1}$ with $|U_j\backslash N(v)|\leq |U_j|/2s$. Therefore, from $|U_{i+1}|\geq (4|B|s)^{|B|-i}>|B|s$ we conclude there is a vertex $v\in U_{i+1}$ such that for any $j\geq i+2$ one has $|U_j\backslash N(v)|\geq |U_j|/2s$. 

Setting $\phi(b_{i+1})=v$ is sufficient to show the induction step. The set of candidates for $b_j$ is decreased from $U_j$ to $U_j\backslash (N(v)\cup \{v\})$ and we have
\[|U_j\backslash (N(v)\cup \{v\})|\geq \frac{1}{2s}|U_j|-1\geq \frac{1}{2s}(4|B|s)^{|B|-i}-1\geq (4|B|s)^{|B|-(i+1)}.\]

Thus, by iterating this procedure until $i=|B|$, we conclude that there is an embedding $\phi:V(H)\rightarrow V(G)$ which induces a copy of $H$, thus completing the proof.
\end{proof}

\hide{Clearly, Proposition~\ref{phase1} with Lemma~\ref{greedy} will give Proposition~\ref{from DRC to victory}, so we are now left to find this set $I$. We conclude this subsection with a few comments about why this phase is more complicated than the analysis above. 

Things are complicated by the fact that $H$ has some edges, meaning that $T_b^{(\phi_0)}$ will lie in the common neighborhood of certain vertices (this is why we need dependent random choice). This ruins everything from being revealed in a global order. Instead, for each $b\in B$, we will fix some ordering $a_1,\dots,a_{|A|}$ of $A$, and make sure that $|T_b|$ grows appropriately. This introduces various technical complications that we handle in the next subsection.}

\subsection{Finding a rich independent set}\label{sub:good independent}

In this section, we prove the following proposition, which serves as the second step in the proof of Theorem~\ref{thm:bounded deg variant}. Recall that $C_H=4ab$.

\begin{restatable}{proposition}{goodindependentsets}\label{prop:good independent sets}
Let $G$ be a $K_{s,s}$-free graph and $X\subseteq V(G)$ a set of at least $(C_H s)^{4a+10}$ vertices in which every $k$-tuple of vertices has at least $(C_H s)^{2 h}$ common neighbours. Then $X$ contains a rich independent set of size~$a$. 
\end{restatable}

The techniques discussed in this section may be of independent interest and therefore we state them in a slightly more general setting than needed. We begin by introducing some terminology. Given a hypergraph $\cH$ and $S\subset V(\cH)$, we define the \textit{degree} of the set $S$ as the number of edges of $\cH$ containing $S$, i.e. $\deg_{\cH}(S) := |\{e\in E(\cH): S\subseteq e\}|$. Furthermore, we say that an edge $e\in E(\cH)$ is \textit{$\delta$-heavy} if it contains a set $S\subset e$ and a vertex $v\in e \setminus S$ for which $\deg_\cH(S\cup \{v\})\ge \delta \deg_\cH(S)$. Otherwise, we say that $e$ is \textit{$\delta$-light}. Finally, we say a hypergraph $\cH$ is \textit{$(\eps,\delta)$-superspread} if at most $\eps e(\cH)$ edges of $\cH$ are $\delta$-heavy. The reason we call such hypergraphs ``superspread" is that a $(0,\delta)$-superspread hypergraph must satisfy $\deg_\cH(S)\le \delta^{|S|}e(\cH)$ for all sets $S\subset V(\cH)$, implying that $\cH$ is ``$\delta$-spread'' in the terminology of the recent breakthrough work \cite{ALWZ21} on the Sunflower conjecture.

The definition of $(\eps, \delta)$-superspread hypergraphs is designed with the following goal in mind. Suppose $\fB$ is a collection of ordered $\ell$-tuples of distinct vertices of $\cH$ with the property that for any $v_1, \dots, v_{\ell-1}\in V(\cH)$, there are at most $s$ vertices $v_\ell$ for which $(v_1, \dots, v_\ell)\in \fB$. Then, we can show that many edges of an $(\eps, \delta)$-superspread hypergraph $\cH$ do not contain any $\ell$-tuples of $\fB$. 

Let us now give an imprecise explanation of how these notions will come into play in our proof. We consider the hypergraph $\cH$ which consists of many independent sets of a given size inside the set $X\subseteq V(G)$. The first step is to turn this hypergraph into a $(\eps, \delta)$-superspread hypergraph by cleaning it. Then, we define a collection of ``bad" tuples $\fB$ by saying that $(v_1, \dots, v_\ell)$ is bad if $N(v_\ell)$ contains too many vertices from the set $S=N(v_1, \dots, v_i)\backslash \bigcup_{j=i+1}^{\ell-1} N(v_j)$, for some fixed index $i$. Thus, for any fixed $v_1, \dots, v_{\ell-1}$ for which the set $S$ is not too small, Claim~\ref{claim:simple} shows that there are at most $s$ vertices $v_\ell$ making the $\ell$-tuple $(v_1, \dots, v_\ell)$ bad. Then, we show that there exists an independent set avoiding all bad tuples, which turns out to be sufficient for the independent set to be rich. 

Now, we start presenting the precise statements and proofs. First, we show that every $r$-uniform hypergraph with sufficiently many edges can be turned into a $(\eps, \delta)$-superspread hypergraph through a cleaning procedure. 

\begin{proposition}\label{prop:hypergraph cleaning}
Let $r \ge a\ge 1$ be fixed integers and let $\eps, \delta\in (0, 1)$ be real numbers. Suppose that $\cH$ is an $r$-uniform hypergraph with $n$ vertices and at least $C_r(\eps\delta)^{-r}n^{a-1}$ edges, where $C_r=r^r2^{r^2}$. Then, there exists an integer $t\geq a$ and a nonempty $t$-uniform $(\eps, \delta)$-superspread hypergraph $\cH'$ on the vertex set $V(\cH)$ such that every edge $e'\in E(\cH')$ is contained in some edge $e\in E(\cH)$.
\end{proposition}

\noindent
When proving this proposition, we may assume that $\cH$ is an $r$-partite hypergraph. Indeed, we can find an $r$-partite subhypergraph with at least $\frac{r!}{r^r}$ proportion of the edges, and passing to such a subhypergraph only effects the constant $C_r$. Then, for an $r$-partite $r$-uniform hypergraph $\cH$, with an $r$-partition given by $V(\cH)=V_1\cup \dots \cup V_r$, and for an index set $I\subset [r]$, we define the \textit{restriction} of $\cH$ to $I$, denoted by $\cH_I$, as the hypergraph with the vertex set $V_I=\bigcup_{i\in I} V_i$ and the edge set $E(\cH_I)=\{e\cap V_I:e\in E(\cH)\}$.

The main idea is to show that, as long as $\cH_I$ is not $(\eps, \delta)$-superspread, one can eliminate an index $i$ from $I$ without decreasing the number of edges of $\cH_I$ significantly. Then, the bound on the number of edges of $\cH$ allows us to show that this density increment procedure stops with some uniformity $t\geq a$. We begin by showing this statement in the form of an auxiliary lemma.

\begin{lemma}\label{lemma:density increment}
Let $\cH$ be an $r$-partite $r$-uniform hypergraph with $n$ vertices which is not $(\eps, \delta)$-superspread. Then, there exists an index $i\in [r]$ such that $e(\cH_{[r]\backslash \{i\}})\geq \frac{\eps \delta}{r2^r}e(\cH)$.
\end{lemma}
\begin{proof}
Since $\cH$ is not $(\eps, \delta)$-superspread, at least an $\eps$-fraction of its edges are $\delta$-heavy. In other words, for at least $\eps e(\cH)$ edges $e\in E(\cH)$, there exists a set $S_e\subset e$ and a vertex $v_e\in e\backslash S_e$ such that $\deg_{\cH}(\{v_e\}\cup S_e)\geq \delta \deg_{\cH} (S_e)$. Let us record for each edge, from which parts $V_1, \dots, V_r$ the vertices of $S_e$ and the vertex $v_e$ come from. Formally, for each $\delta$-heavy edge $e$, we can define the index $i_e\in [r]$ such that $v_e=e\cap V_{i_e}$. Similarly, we define the set $J_e\subseteq [r]$ for which $S_e=e\cap \bigcup_{j\in J_e} V_j$. By the pigeonhole principle, there exist an index set $J=\{j_1, \dots, j_{m}\}\subseteq [r]$ and an index $i\in [r]$ such that $S_e=e\cap V_J$ and $\{v_e\}=e\cap V_i$ for at least an $\frac{\eps}{r 2^r}$-fraction of the edges of $\cH$. Let us denote the set of such edges by $E^*$.

Without loss of generality, we may assume that $J=[m]$ and $i=m+1$. We claim  that $\cH_{[r]\backslash \{m+1\}}$ has at least $\frac{\eps \delta}{r 2^r}e(\cH)$ edges, or equivalently, the edges of $\cH$ form at least $\frac{\eps \delta}{r 2^r}e(\cH)$ distinct intersections with $V(\cH_{[r]\backslash \{m+1\}})=\bigcup_{i\neq m+1} V_i$. The number of edges of the hypergraph $\cH_{[r]\backslash \{m+1\}}$ can be computed as the sum of the degrees of all $m$-tuples of vertices from $V_{1}\times\cdots\times V_{m}$. Formally, we have 
\[e(\cH_{[r]\backslash \{m+1\}})=\sum_{v_{1}\in V_{1}, \dots, v_{m}\in V_m} \deg_{\cH_{[r]\backslash \{m+1\}} }(v_{1}, \dots, v_m).\]
Observe that the degrees $\deg_{\cH_{[r]\backslash \{m+1\}} }(v_{1}, \dots, v_m)$ are lower bounded by \[\deg_{\cH_{[r]\backslash \{m+1\}} }(v_{1}, \dots, v_m)\geq \max_{v\in V_{m+1}}\deg_{\cH}(v_{1}, \dots, v_m, v).\]
Furthermore, for each $\delta$-heavy edge $e\in E^*$ given by $e=(v_1, \dots, v_r)$, we have $\delta\deg_{\cH}(v_1, \dots, v_m)\leq  \deg_{\cH}(v_1, \dots, v_{m+1})$. Thus, for every $m$-tuple $(v_1, \dots, v_m)$ contained in some $\delta$-heavy edge $e\in E^*$, one has 
\[\deg_{\cH_{[r]\backslash \{m+1\}} }(v_{1}, \dots, v_m)\geq \deg_{\cH}(v_{1}, \dots, v_{m+1})\geq \delta \deg_{\cH}(v_{1}, \dots, v_m).\]

Thus, we can lower bound the number of edges in $\cH_{[r]\backslash\{m+1\}}$ by restricting the sum over only those tuples $v_1\in V_1, \dots, v_m\in V_m$ for which there is a $\delta$-heavy edge $e\in E^*$ containing them. Then, we get
\begin{align*}
e(\cH_{[r]\backslash \{m+1\}})&\geq \sum_{\substack{v_{1}\in V_{1}, \dots, v_{m}\in V_m \\ \exists e\in E^*:v_1, \dots, v_m\in e}} \deg_{\cH_{[r]\backslash \{m+1\}} }(v_{1}, \dots, v_m)\\
&\geq \sum_{\substack{v_{1}\in V_{1}, \dots, v_{m}\in V_m \\  \exists e\in E^*:v_1, \dots, v_{m}\in e}} \delta\deg_{\cH }(v_{1}, \dots, v_m)\\
&\geq \delta |E^*|\geq \frac{\eps \delta}{r 2^r} e(\cH).
\end{align*}
This finishes the proof.
\end{proof}

Now we are ready to prove Proposition~\ref{prop:hypergraph cleaning}, which shows that we can restrict every hypergraph to an $(\eps, \delta)$-superspread hypergraph.

\begin{proof}[Proof of Proposition~\ref{prop:hypergraph cleaning}.]
As suggested in the short discussion following the statement of Proposition~\ref{prop:hypergraph cleaning}, by randomly partitioning the vertex set of $\cH$ into $r$ parts, we may pass to a hypergraph $\cH_0\subseteq \cH$ which is an $r$-partite $r$-uniform hypergraph with at least $\frac{r!}{r^r} e(\cH)$ edges. Let $V_1,\dots,V_r$ be the parts of $\cH_0$.

Lemma~\ref{lemma:density increment} shows that if $\cH_0$ is not $(\eps, \delta)$-superspread, one can eliminate one of the parts $V_i$ and obtain an $(r-1)$-partite $(r-1)$-uniform hypergraph $\cH_1$ with at least $\frac{\eps \delta}{r2^r}e(\cH_0)$ edges. In general, this procedure can be repeated, thus obtaining a sequence of hypergraphs $\cH_0, \cH_1, \dots, \cH_m$, where $\cH_i$ is an $(r-i)$-partite $(r-i)$-uniform hypergraph with \[e(\cH_i)\geq \frac{(\eps \delta)^i}{r(r-1)\cdots (r-i+1)}2^{-\sum_{k=r-i+1}^r k}\cdot e(\cH_0)> \frac{(\eps\delta)^{r}}{r!} 2^{-r^2} \left(\frac{r!}{r^r}C_r (\eps \delta)^{-r} n^{a-1}\right)\geq n^{a-1}.\]

Since we have $e(\cH_i)>n^{a-1}$ for all $i$, the process must stop at uniformity $t=r-i\geq a$, as a $t$-uniform hypergraph on $n$ vertices has less than $n^{t}$ edges. Thus, we find a $(\eps, \delta)$-superspread hypergraph $\cH_i$ which satisfies all requirements.
\end{proof}

In the next proposition, we show how to use the fact that $\cH$ is superspread to bound the number of edges of $\cH$ which contain ``bad'' tuples from certain structured collections $\mathfrak{B}$.

\begin{proposition}\label{prop:spread counting}
Let $\cH=(V,E)$ be a $t$-uniform $(\eps,\delta)$-superspread hypergraph and let $\mathfrak{B}$ be a collection of ``bad" ordered $\ell$-tuples of distinct elements of $V(\cH)$ satisfying the property that for any $v_1, \dots, v_{\ell-1}\in V(\cH)$, there exist at most $s$ vertices $v_\ell$ for which $(v_1, \dots, v_\ell)\in \fB$. Then, at most $(\eps+t! s \delta)e(\cH)$ edges of $\cH$ contain a tuple of $\fB$.
\end{proposition}
\begin{proof}
Since $\cH$ is a $(\eps, \delta)$-superspread hypergraph, at most $\eps e(\cH)$ edges of $\cH$ are $\delta$-heavy. Thus, it is enough to show that among the $\delta$-light edges of $\cH$, at most $t! s \delta e(\cH)$ contain a tuple of $\fB$. To this end, say that an edge of $\cH$ is bad if it is $\delta$-light and contains some tuple of $\fB$. Observe that the number of $\delta$-light edges of $\cH$ containing a bad $\ell$-tuple $(v_1, \dots, v_\ell)\in \fB$ can be bounded by $\deg_\cH(v_1, \dots, v_\ell)\leq \delta\deg_{\cH}(v_1, \dots, v_{\ell-1})$, where the inequality follows from the definition of $\delta$-lightness. Furthermore, for any fixed $(\ell-1)$-tuple of distinct vertices $(v_1, \dots, v_{\ell-1})$, there exist at most $s$ vertices $v_\ell$ completing it to a bad $\ell$-tuple $(v_1, \dots, v_{\ell-1}, v_\ell)\in \fB$. Thus, we conclude that the number of bad edges of $\cH$ which contain a fixed $(\ell-1)$-tuple $(v_1, \dots, v_{\ell-1})$ is at most $\delta s \deg_{\cH}(v_1, \dots, v_{\ell-1})$.

Hence, the total number of bad edges can be bounded by \[\delta s\sum_{\text{distinct }v_1, \dots,v_{\ell-1}} \deg_{\cH}(v_1, \dots, v_{\ell-1})=\delta s (\ell-1)!\binom{t}{\ell-1}e(\cH).\] The last equality is the consequence of double counting, as the sum $\sum \deg_{\cH}(v_1, \dots,v_{\ell-1})$ counts the number of pairs $(e, (v_1, \dots, v_{\ell-1}))$ for which the vertices $v_1, \dots, v_{\ell-1}$ belong to $e$ and $e\in E(\cH)$. Since every edge $e$ of cardinality $t$ contains exactly $(\ell-1)!\binom{t}{\ell-1}$ ordered $(\ell-1)$-tuples, the equality follows. Since $\delta s (\ell-1)!\binom{t}{\ell-1}e(\cH)<\delta s t! e(\cH)$, we conclude that the number of bad edges is at most $\delta s t! e(\cH)$, finishing the proof.
\end{proof}

\begin{corollary}\label{cor:spread independent sets}
Let $s, a\ge 1$ be integers. Let $G$ be a graph which does not contain $K_{s,s}$ and let $X\subseteq V(G)$ be a subset of at least $|X|\ge (2a)^{4a+10}(4s)^{4a+2}$ vertices. Then, there exists some $t\in [a,2a]$ and a non-empty collection $\mathcal{I}$ of independent sets of size $t$ in $X$ such that $\mathcal{I}$ is an $(\eps,\delta)$-superspread hypergraph, where $\eps=(2a)^{-2}$ and $\delta=(2a)^{-2(a+1)}s^{-1}$.
\end{corollary}
\begin{proof}
Set $r:= 2a$. The main idea is to apply Proposition~\ref{prop:hypergraph cleaning} to the $r$-uniform hypergraph, whose edges are the independent sets of size $2a$ in $X$. In order to do this, we need to verify that there are at least $C_{r}(\eps\delta)^{-r}|X|^{a-1}$ independent sets of size $r$ in $X$, where $C_{r}=r^{r} 2^{r^2}$. 
    
This will simply follow from a supersaturation argument. Namely, by the Erd\H{o}s-Szekeres theorem \cite{ESz}, the Ramsey number of $K_{s, s}$ versus $K_{2a}$ can be upper bounded as $$R(K_{s, s}, K_{r})\leq R(K_{2s}, K_{r})\leq \binom{2s+r-2}{r-1}\leq (2s)^{r}.$$ Hence, every set $Y\subseteq X$ of size $|Y|= (2s)^{r}$ contains a independent set of size $r$ (since $G$ is $K_{s, s}$-free). On the other hand, each independent set of size $r$ belongs to $\binom{|X|-r}{(2s)^{r}-r}$ sets $Y$ of size $|Y|=(2s)^{r}$. Thus, the collection $\mathcal{I}_{r}$ of independent sets of size $r$ in $X$ has cardinality at least 
\begin{align*}
    |\mathcal{I}_{r}| &\geq \frac{\binom{|X|}{(2s)^{r}}}{\binom{|X|-r}{(2s)^{r}-r}}\geq \left(\frac{|X|}{(2s)^{r}}\right)^{r}\geq 
    \left(\sqrt{|X|}\cdot  \frac{r^{r+5}(4s)^{r+1}}{(2s)^{r}}\right)^{r}\\
    &\geq r^{r} 2^{r^2} r^{2r} \big(sr^{r+2}\big)^{r} |X|^a
    \geq C_{r}(\eps\delta)^{-r} |X|^{a-1}.
\end{align*}
This suffices for Proposition~\ref{prop:hypergraph cleaning} to be applied and therefore we find the collection $\mathcal{I}$ satisfying all necessary conditions.
\end{proof}

Finally, we show how to find rich independent sets in the graph $G$.

\begin{proof}[Proof of Proposition \ref{prop:good independent sets}]

The main idea of this proof is to apply Corollary~\ref{cor:spread independent sets} to find a superspread collection of independent sets in $X$. Then, we define a set of ``bad" tuples with the property that if an independent set contains no ``bad" tuples, then it must be a rich independent set.

Let us present this argument formally. Since $|X|\geq (4|A||B| s)^{4|A|+10}$, Corollary~\ref{cor:spread independent sets} implies that there exists a non-empty $(\eps, \delta)$-superspread collection of independent sets $\cI$ of size $t\in [|A|, 2|A|]$ in $X$, where $\eps=(2a)^{-2}$ and $\delta=(2a)^{-2(a+1)}s^{-1}$. Then, we define the $tk$ collections of bad ordered tuples, $\fB_{i, \ell}$ for $i\in [k]$ and $\ell\in [t]$ as follows. The $\ell$-tuple of distinct vertices $(v_1, \dots, v_\ell)$ belongs to $\fB_{i,\ell}$ if it satisfies the following two conditions:
\begin{itemize}
    \item The set $S=N(v_1, \dots, v_i)\backslash \bigcup_{j=i+1}^{\ell-1} N(v_{j})$ contains at least $2s$ vertices, and
    \item $|S\backslash N(v_\ell)|<\frac{1}{2s}|S|$.
\end{itemize}
By Claim~\ref{claim:simple}, for any $(v_1, \dots, v_{\ell-1})$ satisfying the first condition, there are at most $s$ vertices $v_\ell$ for which $(v_1, \dots, v_\ell)\in \fB_{i, \ell}$. Therefore, $\fB_{i, \ell}$ satisfies the requirement of Proposition~\ref{prop:spread counting}, and we conclude that at most $(\eps+t! s\delta)e(\cI)$ independent sets of $\cI$ contain a tuple of $\fB_{i, \ell}$.  

Thus, there are at most $tk(\eps+t!s\delta)|\cI|$ independent sets of $\cI$ which contain a bad tuple from any of the families $\fB_{i, \ell}$. In particular, noting that $tk(\eps+t!s\delta)\leq 2a^2(\frac{1}{4a^2}+\frac{(2a)! s}{(2a)^{2a+2} s})\leq \frac{1}{2}+\frac{1}{4}<1$, there exists an independent set $I_0\in \cI$ containing no bad tuples. Let $I\subseteq I_0$ be an independent set of size $|A|$.

We claim that $I$ is a rich independent set, that is, for any $T\subset I$ of size at most $k$, there are at least $(4|B|s)^{|B|}$ vertices $v\in V(G)\backslash I$ which are adjacent to all vertices of $T$ and non-adjacent to all vertices of $I\backslash T$. Indeed, consider a subset $T\subset I$ with $|T|=i\le k$ and let us enumerate the vertices of $I$ by $v_1, \dots, v_{|A|}$ such that $T=\{v_1, \dots, v_i\}$. By the definition of $X$, any $k$ vertices of $X$ have at least $(C_H s)^{2|V(H)|}$ common neighbours in $G$ and therefore $|N(v_1, \dots, v_i)|\geq (C_H s)^{2|V(H)|}$. We can now show by induction that \[\left|N(v_1, \dots, v_i)\Big\backslash \bigcup_{j=i+1}^\ell N(v_j)\right|\geq (2s)^{-(\ell-i)} (C_H s)^{2|V(H)|}.\]
For $\ell=i$, this is equivalent to the fact $v_1, \dots, v_i$ have many common neighbours, which was stated above. For $\ell>i$, under the assumption $|N(v_1, \dots, v_i)\backslash \bigcup_{j=i+1}^{\ell-1} N(v_j)|\geq (2s)^{-(\ell-i)+1} (C_H s)^{2|V(H)|}\geq 2s$, by recalling the fact that $I$ contains no bad tuples from $\fB_{i, \ell}$, we conclude that 
\[\left|\Big(N(v_1, \dots, v_i)\big\backslash \bigcup_{j=i+1}^{\ell-1} N(v_j)\Big)\backslash N(v_\ell)\right|\geq \frac{1}{2s} \cdot (2s)^{-(\ell-i+1)} (C_H s)^{2|V(H)|}= (2s)^{-(\ell-i)} (C_H s)^{2|V(H)|}.\]
In particular, when $\ell=|I|$, we conclude that the number of vertices of $G$ adjacent to all vertices of $T$ and nonadjacent to all vertices of $I\backslash T$ is at least $(C_H s)^{|V(H)|}\geq (4|B|s)^{|B|}$. This suffices to conclude that $I$ is a rich independent set, completing the proof.
\end{proof}

\subsection{Finishing the proof}

In this section, we put everything together to prove Theorem \ref{thm:bounded deg variant}. In the proof, we use the following standard form of the dependent random choice lemma.

\begin{lemma}[Lemma 2.1 in \cite{FS}]\label{lemma:dep_rand_choice}
Let $\ell, m, k$ be positive integers. Let $G = (V, E)$ be a graph with $|V| = n$ vertices and average degree $d = 2|E(G)|/n$. If there is a positive integer $t$ such that
$$\frac{d^{t}}{n^{t-1}}-\binom{n}{k}\left(\frac{m}{n}\right)^{t}\geq \ell$$
then $G$ contains a subset $X$ of at least $\ell$ vertices such that every $k$ vertices in $X$ have at least $m$ common neighbors.
\end{lemma}

\begin{proof}[Proof of Theorem \ref{thm:bounded deg variant}]
(i) We follow the outline presented in the beginning of this section. The first step is to apply Lemma~\ref{lemma:dep_rand_choice} with $\ell=(C_Hs)^{4h+10}$, $m=(C_H s)^{2h}$, and $t=k$, and note that $d\geq (C_H s)^{4h+10}n^{1-1/k}$. The assumption of Lemma~\ref{lemma:dep_rand_choice} is satisfied since \[\frac{d^{t}}{n^{t-1}}-\binom{n}{k}\left(\frac{m}{n}\right)^{t}\geq (C_H s)^{(4h+10)k} - \frac{n^k}{k!}\frac{(C_H s)^{2hk}}{n^k} \geq \ell.\] Hence, we can find a subset $X\subset V(G)$ such that $|X|\geq \ell$ and every $k$-tuple of vertices in $X$ has at least $m$ common neighbors. By Proposition \ref{prop:good independent sets}, $X$ contains a rich independent set of size $a$. Finally, by Lemma \ref{lemma:greedy}, $G$ contains an induced copy of $H$.

(ii) The proof is almost the same as above: apply Lemma~\ref{lemma:dep_rand_choice} with $\ell=m=\sqrt{n} \geq (C_H s)^{4h+10}$, and $t=2k$, and note that $d\geq 2n^{1-1/4k}$. Since \[\frac{d^{t}}{n^{t-1}}-\binom{n}{k}\left(\frac{m}{n}\right)^{t}\geq 2^{2k}\frac{n^{2k-\frac{1}{2}} }{n^{2k-1}} - \frac{n^k}{k!} \left(\frac{\sqrt{n}}{n}\right)^{2k}\geq 2^{2k}\sqrt{n}-1\geq \ell,\] the assumption of Lemma~\ref{lemma:dep_rand_choice} holds and we can find a subset $X\subset V(G)$ such that $|X|\geq \ell\geq (C_H s)^{4h+10}$ and every $k$-tuple of vertices in $X$ has at least $m\geq (C_H s)^{2h}$ common neighbors. From this point, the proof is identical to that of (i).
\end{proof}

\subsection{Lower bounds}
To conclude this section, we briefly sketch a lower bound construction for the induced Tur\'an number of a complete bipartite graph. Namely, in Theorem~\ref{thm:bounded deg variant} we have shown that $\ex^*(n, H, s)\leq (C_H s)^{4|V(H)|+10} n^{2-1/k}$, for some constant $C_H$ depending on the graph $H$. Now, we show that the constant next to $n^{2-1/k}$ indeed needs to grow with $s$ polynomially, at least when $H$ is the complete bipartite graph. 

\begin{proposition}
Let $H = K_{k, \ell}$ where $\ell= (k-1)!$. Then, for all integers $n\geq s\geq 2k!$ we have \[\ex^*(n, H, s) = \Omega_k( s^{1/k} n^{2-1/k}).\]
\end{proposition}
\begin{proof}
Let $G_0$ be an extremal bipartite $H$-free graph on $n_0=\frac{n}{t}$ vertices, where $t=\frac{s}{2k!}$. By a classical result of Alon, R\'onyai and Szab\'o \cite{ARS99} (see also \cite{Bukh}), there exists such $G_0$ with at least $\Omega(n_0^{2-1/k})$ edges. Then, we let $G$ be a $t$-fold blowup of the graph $G_0$, where every vertex of $G_0$ is replaced by a clique of size $t$ and every edge is replaced by a complete bipartite graph $K_{t, t}$. We say that two vertices in such a $t$-clique are \emph{twins}. The graph $G$ has at most $n_0t\leq n$ vertices and at least $\Omega(t^2 e(G_0))=\Omega( t^{1/k} (n_0t)^{2-1/k} )=\Omega_k(s^{1/k} n^{2-1/k})$ edges. If we show that $G$ has no induced copy of $H$ and no $K_{s, s}$, this implies that $\ex^*(n, H, s)= \Omega_k(s^{1/k}n^{2-1/k})$. 

Let us observe that if $U\subset V(G)$ induces a copy of $H$, then no two vertices of $U$ are twins. If two vertices from the same part of $H$ are embedded into twins, they are adjacent in $G$, which is not possible. On the other hand, if two vertices from different parts of $H$ are embedded into twins, then any other vertex of $G$ is either connected to both of them or to none. Therefore, we conclude that $U$ contains no twins. But then each vertex of $U$ corresponds to a unique vertex in $G_0$, so $G_0$ contains an induced copy of $H$ as well, contradiction.

Finally, one can argue that if $G$ contains $K_{s, s}$ as a subgraph, then $G_0$ must contain $K_{s/t, s/t}=K_{2k!,2k!}$, contradiction. This completes the proof.
\end{proof}

One can slightly improve the above construction when $n=2^{O(s)}$ by replacing vertices in the blow-up by Ramsey graphs avoiding $K_{s/2, s/2}$ and independent sets of size $h$, instead of replacing them by cliques. If one replaces the edges between these vertices with dense graphs which avoid $K_{s/2, s/2}$, the resulting graph has no $K_{s, s}$. Hence, this improved construction shows that $\ex^*(n, H, s)\geq s^{|V(H)|^{1/(k+1)}}n^{2-\frac{1}{k}}$, i.e., that the exponent of $s$ should grow with $|V(H)|$. We omit further details since there is still a large gap between upper and lower bounds.

\section{Trees, induced $C_4$-free subgraphs and the Erd\H{o}s-Hajnal conjecture}\label{sec:applications}

In this section, we prove Theorem \ref{thm:trees}. We also present three short applications of the results obtained in the previous section, proving Theorems~\ref{thm:induced C_4-free}--\ref{thm:half Erdos Hajnal}. 

Let us begin by discussing Theorem~\ref{thm:trees}. The starting point of our proofs is Proposition 7.9 from a recent paper of Gir\~ao and the first author \cite{GH23}, which shows that within every graph of large enough average degree, one can find either an induced subgraph which is $C_4$-free with average degree larger than any constant or an induced subgraph with almost quadratically many edges. 

\begin{proposition}[\cite{GH23}]\label{prop:GHblackbox}
Let $H$ be a bipartite graph and fix an integer $k$ and $\eps>0$. For all sufficiently large $d$, the following is true. Let $G$ be a graph with no induced copy of $H$ and with average degree $d$. Then $G$ has an induced subgraph $G'$ satisfying one of the following: 
\begin{itemize}
    \item $G'$ has no $C_4$ and has average degree at least $k$;
    \item $G'$ has $n'\ge d^{1/10}$ vertices and $e(G')\ge (n')^{2-\eps}$.
\end{itemize}
\end{proposition}

\noindent
We note that the \textit{sufficiently large} condition in the above result is not especially quantitative. It would be interesting to determine whether we can take $d\leq k^{O_H(1/\eps)}$.

In either of the cases, our goal is to find an induced copy of our tree $T$ within $G'$. Before we embark on our proof, we first show that it is easy to embed $T$ in $C_4$-free graphs of large average degree. 

\begin{proposition}\label{prop:trees base bound}
Let $T$ be a tree on $t$ vertices. Then $\ex^*(n,T,2)\le 2tn$.
\end{proposition}
\begin{proof}
We proceed by induction on the number of vertices, and we note that the statement is trivial when $T$ consists of a single vertex. For general trees, let us fix a host graph $G$ on $n$ vertices which is $C_4$-free and has $e(G)> 2tn$.

We may clean the graph $G$ such that its minimal degree is at least $2t$. Namely, as long as there exists a vertex of degree less than $2t$ in $G$, one may remove it. The maximum number of edges removed in this way is at most $2tn$ and therefore the cleaning procedure ends with a non-empty graph whose minimal degree is at least $2t$. To simplify notation, we still denote this cleaned version of the graph by $G$.

Let $T'$ be a tree on $t-1$ vertices obtained by deleting a leaf $v$ from the tree $T$. We denote by $u$ the unique neighbour of $v$ in $T$. By inductive hypothesis, the graph $G$ contains an induced copy of $T'$.

If the vertex $u'$ plays the role of $u$ in this copy, if suffices to find a vertex $v'$ in the neighbourhood of $u'$ which is not adjacent to any of the other vertices in the induced copy of $T'$. 

Since $G$ is $C_4$-free, $u'$ shares at most one neighbour with any other vertex of this induced copy of $T'$. Therefore, since the neighbourhood of $u'$ has at least $2t$ elements, it contains a vertex $v'$ different from and non-adjacent to all other vertices of the induced copy of $T'$. Adding this vertex to the copy of $T'$ gives an induced copy of $T$, thus completing the proof.
\end{proof}

\begin{proof}[Proof of Theorem~\ref{thm:trees}.]
Let $G$ be a graph with $n$ vertices and at least $s^{Ct}n$ edges, where $C=150$. We show that if $G$ is $K_{s, s}$-free, then there exists an induced copy of the tree $T$ on $t$ edges in $G$. Let us assume, for the sake of contradiction, that $G$ does not contain an induced copy of $T$ and further assume that $t\geq 3$, since otherwise the statement is trivial. 

Apply Proposition~\ref{prop:GHblackbox} to the graph $G$ with parameters $H:=T$, $k:=4t+1$ and $\eps:=\frac{1}{4t}$. We conclude that, for all sufficiently large $s$, either $G$ contains an induced $C_4$-free subgraph $G'\subseteq G$ with average degree greater than $4t$, or an induced subgraph $G'\subseteq G$ on $n'\geq s^{Ct/10}$ vertices and average degree $(n')^{1-\frac{1}{4t}}$.

In the first case, we can apply Proposition~\ref{prop:trees base bound} to conclude that $G'$ contains an induced copy of $T$, and hence $G$ also contains an induced copy of $T$. In the second case, we apply (ii) of Theorem~\ref{thm:bounded deg variant} to the host graph $G'$ with $n'\geq s^{Ct/10}>(t^2s)^{8t+20}$ vertices and $e(G')\ge (n')^{2-\frac{1}{4t}}$ edges. Since $G$ is $K_{s, s}$-free, so is $G'$, and we conclude that $G'$ contains an induced copy of $T$. This completes the proof.
\end{proof}

In what follows, we present further applications of Theorem~\ref{thm:bounded deg variant}.

\begin{proof}[Proof of Theorem \ref{thm:induced C_4-free}]
For each $k\geq 1$, we construct a $C_4$-free bipartite graph $H_k$ of average degree $k$ and at most $8k^2$ vertices. To do this, we pick the smallest prime power $q$ satisfying $q\geq k-1$ and setting $H_k$ to be the point-line incidence graph of projective plane with $q^2+q+1$ elements. Then, $H_k$ is a $(q+1)$-regular graph and so $H$ is a $C_4$-free graph with average degree at least $k$. Moreover, by Bertrand's postulate, we have $q\leq 2(k-1)$ and thus $H_k$ has at most $2(q^2+q+1)\leq 8k^2$ vertices. 

Let us now suppose, for the sake of contradiction, that for some sufficiently large $s$, there exists a $K_{s, s}$-free graph $G$ with average degree $d(G)\ge s^{10^3k^2}$ and no $C_4$-free induced subgraph of average degree at least $k$. Since $G$ does not contain an induced copy of $H_k$ and $s$ is sufficiently large compared to $k$, we may apply Proposition~\ref{prop:GHblackbox} to it with parameters $k$ and $\eps=\frac{1}{8k}$.

By our assumption on $G$, the first conclusion of the Proposition~\ref{prop:GHblackbox} cannot hold and therefore $G$ contains an induced subgraph $G'\subseteq G$ on $n'\geq d(G)^{1/10}\geq s^{100 k^2}>(C_{H_k} s)^{8|V(H_k)|+20}$ vertices and with $e(G')\geq (n')^{2-\frac{1}{8k}}$ edges. Since $H_k$ is a $(q+1)$-regular graph, its degrees are bounded by $2k$ and therefore part (ii) of Theorem~\ref{thm:bounded deg variant} shows that $G'$ contains an induced copy of $H_k$. This presents a contradiction to the assumption that $G$ has no $C_4$-free induced subgraph of average degree at least $k$, thus completing the proof.
\end{proof}

\begin{proof}[Proof of Theorem~\ref{thm:induced subdivisions turan}.]

Assume that the statement is not true and there exists a $K_{s, s}$-free graph $G$ which does not contain an induced subdivision of $H$ and has average degree at least $s^{150|V(H)|}$. Also, assume $|V(H)|\geq 3$, since the statement is otherwise trivial. A result of K\"uhn and Osthus \cite{KO04} states that every $C_4$-free graph $G'$ of large enough average degree $d(G')\geq d_0$ must contain an induced subdivision of $H$, where $d_0$ may depend on $H$. Therefore, we conclude $G$ does not contain an induced $C_4$-free subgraph $G'$ of average degree at least $d_0$.

Furthermore, since $G$ does not contain an induced subdivision of $H$, it does not contain an induced copy of $H$. Therefore, Proposition~\ref{prop:GHblackbox} applies to $G$, with the parameters $k=d_0$ and $\eps=\frac{1}{4|V(H)|}$. By the above discussion, the first conclusion of Proposition~\ref{prop:GHblackbox} cannot hold and hence $G$ contains an induced subgraph $G'$ on $n'\geq d(G)^{1/10}\geq s^{15|V(H)|}\geq (C_H s)^{8|V(H)|+20}$ vertices and with $e(G')\geq (n')^{2-1/4|V(H)|}$ edges. But Theorem~\ref{thm:bounded deg variant} now implies that $G$ contains either $K_{s, s}$ or an induced copy of $H$, contradiction.
\end{proof}

\begin{proof}[Proof of Theorem \ref{thm:half Erdos Hajnal}] Set $s=\frac{1}{h^2}n^{\frac{1}{8h+20}}$. If $G$ does not contain $K_{s, s}$ as a subgraph and $G$ does not contain an induced copy of $H$, we apply Theorem~\ref{thm:bounded deg variant} to deduce that $G$ has few edges. By our choice of $s$, we have $n\geq (C_H s)^{8h+20}$ and therefore part (ii) of Theorem~\ref{thm:bounded deg variant} applies to show that $G$ has at most $n^{2-\frac{1}{4h}}$ edges. Hence, $G$ must contain an independent set of size at least $n^{\frac{1}{4h}}$, which completes the proof.
\end{proof}

\section{Cycles and the cube}\label{sec:cycles}

The goal of this section is to prove Theorem~\ref{thm:cycles}, which is concerned with graphs that are induced $C_{2k}$-free and $K_{s,s}$-free. Also, we show Proposition~\ref{prop:cube} which gives an upper bound on the induced Tur\'an number of the cube graph. We begin by proving the following strengthening of Theorem~\ref{thm:cycles}, which is needed to for our proof of Proposition~\ref{prop:cube}.

\begin{theorem}\label{thm:alternating cycles}
Let $k, s\geq 2$ be integers and let $G$ be a graph on $n$ vertices which does not contain $K_{s, s}$ as a subgraph. Then there exists a constant $C=C_{k, s}>0$ such that for any partition $V(G)=A\cup B$ with at least $e(A, B)\geq Cn^{1+\frac{1}{k}}$ crossing edges, there exists an induced copy of the cycle $C_{2k}$ on vertices $v_1, v_2, \dots, v_{2k}$ such that $v_1, v_3, \dots, v_{2k-1}\in A$ and $v_2, v_4, \dots, v_{2k}\in B$.
\end{theorem}

Since any graph $G$ contains an partition with at least $e(G)/2$ crossing edges, Theorem~\ref{thm:cycles} is directly implied by Theorem~\ref{thm:alternating cycles}.

Before presenting the main idea of the proof, we introduce some terminology and notation which will be used throughout the section. For every graph $G$ with partition $V(G)=A\cup B$, we define the subgraph $G_1$ containing the crossing edges. Furthermore, we say that a cycle of length $2k$ in $G$ is \textit{alternating} if all edges of this cycle cross the partition $A, B$. Similarly, a path is \textit{alternating} if all of its edges cross the partition.

Now, we present the outline of our proof. The first step is to pass to subsets $A_0\subseteq A, B_0\subseteq B$ such that $G_1[A_0\cup B_0]$ is an almost regular graph. To do this, we use the notion of $\alpha$-maximality, introduced in \cite{Tom23}. Then, we use a result of Janzer \cite{J23} to show that most homomorphic cycles of length $2k$ are non-degenerate. Finally, we argue that if there are no induced cycles of length $2k$ in $G$, then one can choose vertices $u, v\in V(G)$ with many disjoint paths between them and many edges between internal vertices of these paths. This will allow us to find a dense subgraph of $G$, in which we find a copy of $K_{s, s}$, using the K\H{o}v\'ari-S\'os-Tur\'an theorem.

As mentioned above, the first step is to show how to pass from a graph to an almost regular induced subgraph. Given a graph $G$, $\Delta(G)$ denotes the maximum degree, $\delta(G)$ the minimum degree, and $d(G)$ the average degree of $G.$ We say that $G$ is $K$-almost-regular if $\Delta(G)\leq K\delta(G)$.

\begin{lemma}\label{lemma:passing to an almost regular subgraph}
Let $\alpha>0$ be a fixed real number and let $K=2^{3/\alpha+5}$. Furthermore, let $G$ be a graph on $n$ vertices with at least $Cn^{1+\alpha}$ edges. Then, there exists a $K$-almost-regular induced subgraph $H\subseteq G$ on $m$ vertices with at least $\frac{C}{4} m^{1+\alpha}$ edges.
\end{lemma}
\begin{proof}
The main idea of the proof is to pass to a so called $\alpha$-maximal subgraph of $G$ and clean it to remove the low-degree and the high-degree vertices. We say that a graph $G$ is \emph{$\alpha$-maximal} if for any subgraph $H\subseteq G$ one has \[\frac{e(G)}{v(G)^{1+\alpha}}\geq \frac{e(H)}{v(H)^{1+\alpha}}.\]
If $G$ is not $\alpha$-maximal, we may replace $G$ by the subgraph $G'\subseteq G$ maximizing the ratio $\frac{e(G')}{v(G')^{1+\alpha}}$. Note that $G'$ is induced, $\alpha$-maximal, and satisfies $e(G')\geq Cv(G')^{1+\alpha}$. Hence, in what follows, we assume that $G$ is $\alpha$-maximal and set $C_0:=e(G)/v(G)^{1+\alpha}\geq C$. 

Let $U$ be the set of vertices $v\in G$ whose degree is at least $K_0 d(G)$, where $K_0=2^{3/\alpha+2}$. Then, $U$ has size at most $|U|\leq \frac{n}{K_0}$. Using the assumption that $G$ is $\alpha$-maximal, we show that at most a half of the edges in $G$ are incident to $U$.

Suppose this is not the case and we have at least $e(G)/2$ edges incident to $U$. Our goal is to find a set $V\subseteq V(G)\backslash U$ for which $e(G[U\cup V])>C_0 |U\cup V|^{1+\alpha}$, which contradicts the $\alpha$-maximality of $G$. Let $V$ be a random subset of $V(G)\backslash U$ of cardinality $\frac{n}{K_0}$. Then we have $|U\cup V|\leq |U|+|V|\leq \frac{2n}{K_0}$. Furthermore, every edge incident to $U$ belongs to $G[U\cup V]$ with probability at least $\frac{1}{K_0}$ and therefore the expected number of edges induced on $U\cup V$ is at least $\bE[e(G[U\cup V])]\geq e(G)/2K_0$. Therefore, there exists a set $V$ of size $\frac{n}{K_0}$ for which $e(G[U\cup V])\geq e(G)/2K_0$.

We claim that the subgraph $G[U\cup V]$ contradicts the $\alpha$-maximality of $G$, which is verified by the following simple computation:
\begin{align*}
&e(G[U\cup V])-C_0|U\cup V|^{1+\alpha}\geq \frac{e(G)}{2K_0} - C_0\left(\frac{2n}{K_0}\right)^{1+\alpha}\geq \frac{C_0n^{1+\alpha}}{2K_0} - C_0\left(\frac{2n}{K_0}\right)^{1+\alpha}\geq 0,
\end{align*}
where the last inequality follows from the choice of $K_0$.

Thus, we conclude that $U$ is incident to at most $e(G)/2$ edges and so $V(G)\backslash U$ induces at least $e(G)/2$ edges. Let us now focus on the subgraph $G'\subseteq G$ induced on the set $V(G)\backslash U$. We may repeatedly remove from this graph all vertices of degree at most $d(G')/4$, while removing at most half of the edges of $G'$. In this way, we obtain an induced subgraph $G''\subseteq G$ with the property that $e(G'')\geq e(G')/2\geq e(G)/4$ and $\delta(G'')\geq d(G')/4\geq d(G)/8$. Since all vertices of $G'$ have degree at most $K_0 d(G)$, we finally arrive at the conclusion $8K_0\delta(G'')\geq \Delta(G'')$, thus completing the proof.
\end{proof}

The next step of the proof is to show that almost all homomorphic $2k$-cycles are non-degenerate. Here, a homomorphic $2k$-cycle denotes a sequence of $2k$ vertices, $v_1, \dots, v_{2k}$ such that $v_{i}$ is adjacent to $v_{i+1}$ for all $i=1, \dots, 2k-1$ and $v_1$ is adjacent to $v_{2k}$. Such a cycle is \textit{non-degenerate} if all $2k$ vertices are distinct and otherwise it is \textit{degenerate}. The number of homomorphic $2k$-cycles in $G$ is denoted by $\hom(C_{2k}, G)$. For brevity, throughout this section, we will often refer to homomorphic $2k$-cycles simply as $2k$-cycles.

\begin{lemma}\label{lemma:almost all cycles are nondegenerate}
Let $k$ be a positive integer, let $K=2^{3k+5}$ and let $G$ be a $K$-almost-regular graph with $d(G)= Cn^{1/k}$. Then at most $\frac{2^{2k+10}}{\sqrt{C}}\hom(C_{2k}, G)$ homomorphic $2k$-cycles in $G$ are degenerate.
\end{lemma}

The key ingredient in the proof of this lemma is the following result of Janzer \cite{J23}, which controls the number of degenerate $2k$-cycles in a graph.

\begin{lemma}[Lemma 2.2 in \cite{J23}]\label{lemma:janzer}
Let $k \geq 2$ be an integer and let $G = (V, E)$ be a graph on $n$ vertices. Let $\sim$ be
a symmetric binary relation defined over $V$ such that for every $u \in V$ and $v \in V$, $v$ has at most $t$ neighbours $w \in V$ which satisfy $u \sim w$. Then the number of homomorphic $2k$-cycles
$(x_1, x_2, \dots, x_{2k})$ in $G$ such that $x_i \sim x_j$ for some $i\neq j$ is at most
\[32k^{3/2}t^{1/2}\Delta(G)^{1/2}n^{\frac{1}{2k}}\hom(C_{2k}, G)^{1-\frac{1}{2k}}.\]
\end{lemma}

Moreover, we also use the well known fact that even cycles satisfy Sidorenko's conjecture \cite{Sidorenko}.

\begin{lemma}[\cite{Sidorenko}]\label{lemma:sidorenko}
For every graph $G$,
$$\hom(C_{2k},G)\geq d(G)^{2k}.$$
\end{lemma}

\begin{proof}[Proof of Lemma~\ref{lemma:almost all cycles are nondegenerate}.]
We apply Lemma~\ref{lemma:janzer} and define the relation $\sim$ such that $u\sim v$ if and only if $u=v$. Then Lemma~\ref{lemma:janzer} gives a bound on the number of degenerate copies of $2k$-cycles. Since one can take $t=1$, the number of degenerate homomorphic $2k$-cycles in $G$ is at most $$32k^{3/2}\Delta(G)^{1/2}n^{\frac{1}{2k}}\hom(C_{2k}, G)^{1-\frac{1}{2k}}.$$

From Lemma \ref{lemma:sidorenko}, we have $\hom(C_{2k}, G)^{1/2k}\geq d(G)$. Furthermore, since $G$ is a $K$-almost-regular graph, we have $\Delta(G)\leq Kd(G)$. Thus, we have 
\begin{align*}\frac{32k^{3/2}\Delta(G)^{1/2}n^{\frac{1}{2k}}\hom(C_{2k}, G)^{1-\frac{1}{2k}}}{\hom(C_{2k}, G)}&=\frac{32k^{3/2}\Delta(G)^{1/2}n^{\frac{1}{2k}}}{\hom(C_{2k}, G)^{\frac{1}{2k}}}\leq 32k^{3/2}K^{1/2}\frac{d(G)^{1/2} n^{\frac{1}{2k}}}{d(G)}\\
&\leq 32k^{3/2}2^{3k/2+2}\frac{n^{\frac{1}{2k}}}{d(G)^{1/2}}\leq \frac{2^{2k+10}}{\sqrt{C}}.
\end{align*}
\end{proof}

Before we begin the proof of Theorem~\ref{thm:alternating cycles}, we present an auxiliary lemma which shows that if one has a sparse and a dense graph on the same vertex set, one can find a large independent set of the sparse graph in which the density of the dense graph does not decrease too much. This lemma will be applied to an auxiliary graph which is constructed in the course of the main proof.

\begin{lemma}\label{lemma:auxiliary graph lemma}
For every integer $t$ and every $c\in (0, 1)$, there exists $\delta=\delta(t,c)\in (0, 1)$ such that for all $n>c\delta^{-1}$ the following statement holds.
Let $V$ be a set of $n$ vertices and let $G_R, G_B$ be a red and a blue graph on $V$. If $G_R$ has at most $\delta n^2$ edges and $G_B$ has at least $cn^2$ edges, then there exists a set $S\subseteq V$ of size at least $t$ such that $G_B[S]$ has at least $\frac{c}{2}|S|^2$ edges and $G_R[S]$ is empty. 
\end{lemma}
\begin{proof}
Let $S$ be a random subset of $V$ which includes each element with probability $p=\frac{c\delta^{-1}}{10n}$. Furthermore, denote by $e_R(S)$ and $e_B(S)$ the number of red and blue edges in $S$, respectively. 

The first step of the proof is to show that we have the following inequality between expectations
\begin{equation}\label{eqn:ineq expectation}
\bE\big[e_B(S)\big]>\frac{c}{2}\bE\big[|S|^2\big] + \bE\big[|S|\cdot e_R(S)\big] + t^2.
\end{equation}
 For a vertex $u\in V$, we define $\mathbf{1}_u$ to be the indicator function of the event $u\in S$. Let us begin by computing the expectation of $e_B(S)$:
\[\bE[e_B(S)]=\bE\bigg[\sum_{uv\in E(G_B)}\mathbf{1}_u\mathbf{1}_v\bigg]=\sum_{uv\in E(G_B)} \bE[\mathbf{1}_u\mathbf{1}_v]=e(G_B)p^2.\]
Similarly, one can compute the expectation of $|S|^2$ and $|S|\cdot e_R(S)$:
\begin{align*}
\bE\big[|S|^2\big]=\bE\bigg[\Big(\sum_{u\in V}\mathbf{1}_u\Big)^2\bigg]= \bE\bigg[\sum_{u, v\in V}\mathbf{1}_u\mathbf{1}_v\bigg]&=\sum_{u\neq v}\bE[\mathbf{1}_u\mathbf{1}_v]+\sum_{u\in V} \bE[\mathbf{1}_u]=(n^2-n)p^2+np,\\
\bE\big[|S|\cdot e_R(S)\big]=\bE\bigg[\sum_{u,v\in E(G_R)}\mathbf{1}_u\mathbf{1}_v\sum_{w\in V}\mathbf{1}_w\bigg]&=\sum_{\substack{uv\in E(G_R)\\ w\neq u,v}} \bE[\mathbf{1}_u\mathbf{1}_v\mathbf{1}_w]+\sum_{\substack{uv\in E(G_R)\\ w\in \{u,v\}}}\bE[\mathbf{1}_u\mathbf{1}_v]\\
&=e(G_R)(n-2)p^3+2e(G_R)p^2.
\end{align*}
Having computed these, it is straightforward to verify inequality~(\ref{eqn:ineq expectation}):
\begin{align*}
\frac{c}{2}\bE\big[|S|^2\big] + \bE\big[|S|\cdot e_R(S)\big] + t^2&\leq \frac{c}{2}(np)^2+\frac{c}{2}np+\delta n^3p^3+2\delta n^2p^2+t^2 \\
&= \frac{c^3\delta^{-2}}{200}+\frac{c^2\delta^{-1}}{20}+\frac{c^3\delta^{-2}}{1000}+\frac{2c^2\delta^{-1}}{100}+t^2< \frac{c^3\delta^{-2}}{100} \leq e(G_B) p^2,
\end{align*}
assuming $\delta$ is sufficiently small with respect to $t$ and $c$. Therefore, one can choose $S$ such that $e_B(S)> \frac{c}{2}|S|^2 + |S|e_R(S) + t^2$. Since $e_B(S)<\frac{|S|^2}{2}$ for all $S$, we must also have $e_R(S)<|S|/2$. As there are fewer red edges in $S$ than vertices, one can remove one vertex from every red edge and obtain a non-empty set of vertices $S'$ of cardinality $|S'|\geq |S|-e_R(S)$.

We claim that $|S'|\geq t$ and that $S'$ contains at least $\frac{c}{2}|S|^2$ blue edges. Showing both of these inequalities is sufficient to complete the proof by choosing the set $S'$. The first inequality follows from $\frac{|S|^2}{2}>e_B(S)>|S|e_R(S) + t^2$, which implies $$\frac{1}{2}|S'|^2\geq \frac{1}{2}(|S|-e_R(S))^2=\frac{|S|^2}{2}-|S|e_R(S)+\frac{e_R(S)^2}{2}\geq t^2.$$ On the other hand, by removing at most $e_R(S)$ vertices of $S$, one can remove at most $e_R(S)|S|$ blue edges from $S$, meaning that at least $e_B(S)-|S|e_R(S)\geq \frac{c}{2}|S|^2$ blue edges remain in $S'$. This suffices to complete the proof.
\end{proof}

Having covered all the preliminaries, we are ready for the proof of the main theorem.

\begin{proof}[Proof of Theorem~\ref{thm:alternating cycles}.]
Let us begin by applying Lemma~\ref{lemma:passing to an almost regular subgraph} with $\alpha=\frac{1}{k}$ to the graph $G_1$ of edges crossing the partition $(A, B)$. Let $C_0=\frac{C}{4}$ and $K=2^{3k+5}$, then there is a $K$-almost-regular induced subgraph $H_1$ of $G_1$ satisfying $e(H_1)\geq C_0m^{1+1/k}$, where $m$ is the number of vertices of $H_1$. Also, let $H$ be the subgraph of $G$ induced by $V(H_1)$.

For any two distinct vertices $u, v\in V(H)$, we denote by $\cP_{u, v}$ the set of alternating paths of length $k$ between vertices $u$ and $v$ in $H$, and let $P_{u, v}=|\cP_{u, v}|$. Furthermore, let $A_{u, v}$ denote the number of ordered pairs of paths $(P_1, P_2)\in \cP_{u, v}^2$ which intersect in a vertex different from $u, v$. Finally, let $B_{u, v}$ denote the number of ordered pairs of distinct paths $(P_1, P_2)\in \cP_{u, v}^2$ such that some internal vertex of $P_1$ and some internal vertex of $P_2$ are connected by an edge of $H$.

Let us begin by presenting a set of simple observations about these quantities. Since every pair of paths with the same endpoints can be glued to form a homomorphic $2k$-cycle, we have that $\sum_{u,v} P_{u, v}^2\leq \hom(C_{2k}, H_1)$. Here and later, $\sum_{u,v}$ denotes the sum over all pairs $(u,v)\in V(H)^2$, $u\neq v$. Furthermore, every non-degenerate $2k$-cycle $(x_1, \dots, x_{2k})$ can be uniquely partitioned into two paths of length $k$ with the same endpoints, namely the paths $(x_1, x_2, \dots, x_{k+1})$ and $(x_1, x_{2k}, \dots, x_{k+1})$. Hence, the sum $\sum_{u,v} P_{u, v}^2$ is at least the number of non-degenerate $2k$-cycles. For sufficiently large $C_0$, Lemma~\ref{lemma:almost all cycles are nondegenerate} guarantees that at least half of all homomorphic $2k$-cycles in $H_1$ are nondegenerate, and so $\sum_{u, v} P_{u, v}^2\geq \frac{1}{2}\hom(C_{2k}, H_1)$.

The sum $\sum_{u,v} A_{u, v}$ can be bounded by the number of degenerate $2k$-cycles, since each pair of paths which share the endpoints and intersect in some internal vertex corresponds to a degenerate $2k$-cycle. Thus, Lemma~\ref{lemma:almost all cycles are nondegenerate} implies $\sum_{u, v} A_{u, v}\leq \frac{2^{2k+10}}{\sqrt{C_0}}\cdot \hom(C_{2k}, H_1)$.

Finally, if $H$ has no induced nondegenerate $2k$-cycles, then $\sum_{u, v} B_{u, v}\geq \frac{1}{4k} \hom(C_{2k}, H_1)$. To see why, note that every nondegenerate cycle $(x_1, \dots, x_{2k})$ must have a chord, since it is not induced. Then, there exists a way to shift the indices of the vertices, say by setting $y_i=x_{i+t\bmod 2k}$ for some $t$, such that the chord goes between the vertex sets $\{y_2, \dots, y_{k}\}$ and $\{y_{k+2}, \dots, y_{2k}\}$. In this case, the cycle $(y_1, \dots, y_{2k})$ corresponds to a pair of paths between $y_1$ and $y_{k+1}$ with an edge between their internal vertices, which is counted in $B_{y_1, y_{k+1}}$. This argument shows that for any non-degenerate $2k$-cycle, one of its $2k$ cyclic relabeling corresponds to a pair of paths counted in the sum $\sum_{u, v} B_{u, v}$. Hence, 
$2k\sum_{u, v} B_{u, v}$ is at least the number of nondegenerate $2k$-cycles, and so $\sum_{u, v} B_{u, v}\geq \frac{1}{4k}\cdot \hom(C_{2k}, H_1)$.

Combining these observations with the fact that $\hom(C_{2k}, H_1)\geq d(H_1)^{2k}\geq m^2 C_0^{2k}\geq \sum_{u, v}C_0^{2k}$, which comes from Lemma \ref{lemma:sidorenko}, we obtain the following inequality:
\[8k\sum_{u,v}B_{u, v}-\frac{\sqrt{C_0}}{2^{2k+10}} \sum_{u,v} A_{u, v}\geq \hom(C_{2k}, H_1)\geq \frac{1}{2}d(H_1)^{2k}+\frac{1}{2}\sum_{u,v}P_{u, v}^2\geq \frac{1}{2}\sum_{u, v}(P_{u, v}^2+C_0^{2k}).\]
Hence, there exist distinct vertices $u, v$ for which $$8kB_{u, v}\geq \frac{\sqrt{C_0}}{2^{2k+10}} A_{u, v}+\frac{1}{2}P_{u, v}^2+\frac{1}{2}C_0^{2k}.$$ Let us fix these vertices and define $\delta=2^{2k+10}\cdot {8k}/{\sqrt{C_0}}$. Since $B_{u, v}$ counts certain pairs of paths in $\cP_{u, v}$, we always have $B_{u, v}\leq P_{u, v}^2$. Thus, we  deduce that $8k\delta^{-1} A_{u, v}\leq 8kB_{u, v}\leq 8kP_{u, v}^2$ and so $A_{u, v}\leq \delta P_{u, v}^2$. Also, we directly have both $B_{u, v}\geq \frac{1}{16k}P_{u, v}^2$ and $8kP_{u, v}^2\geq 8kB_{u, v}\geq \frac{1}{2}C_0^{2k}$, implying $P_{u, v}\geq C_0^k/4\sqrt{k}$.

We now define two auxiliary graphs, a red and a blue one, on the set of vertices $\cP_{u, v}$. A pair of distinct paths $P_1, P_2\in \cP_{u, v}$ is a red edge if the paths $P_1, P_2$ intersect in some of their internal vertices. Also, $P_1P_2$ is a blue edge if there exists an edge of $H$ between the internal vertices of $P_1$ and $P_2$. Let us denote the red graph by $G_R$ and the blue graph by $G_B$. Note that an ordered pair of paths $(P_1, P_2)$ is counted in $A_{u, v}$ precisely if $P_1P_2$ is a red edge or if $P_1=P_2$. Thus, we have that the number of red edges is precisely $e(G_R)=(A_{u, v}-P_{u, v})/{2}\leq \delta P_{u,v}^2$. Similarly, we have $e(G_B)=B_{u, v}/{2}\geq \frac{1}{32k} P_{u, v}^2$, since $B_{u, v}$ only counts pairs of distinct paths. 

Apply Lemma~\ref{lemma:auxiliary graph lemma} to the graphs $G_R, G_B$ with the parameters $n=P_{u,v}$, $c=\frac{1}{32k}$, and $t=(64k)^{3s}$. The conditions of the Lemma~\ref{lemma:auxiliary graph lemma} hold when $C_0$ is sufficiently large, since $e(G_R)\leq \delta P_{u, v}^2$ and $P_{u, v}>c\delta^{-1}$ as $P_{u, v}\geq C_0^k/4\sqrt{k}$. We conclude that there exist $r\geq t$ paths $P_1, \dots, P_r\in \cP_{u, v}$ with disjoint interiors and with at least $\frac{1}{64k} r^2$ pairs $P_i, P_j$ having an edge of $H$ between their internal vertices.

Let us now consider an induced subgraph $F\subseteq H$ on the union of vertices of all these paths. The number of vertices of $F$ is $h=2+r(k-1)$ and the number of edges in $F$ is at least $\frac{1}{64k} r^2$. However, $F$ is a $K_{s, s}$-free graph and the K\H{o}v\'ari–S\'os–Tur\'an theorem implies that $e(F)\leq (s/h)^{1/s}h^2+hs/2$. Since $r\geq t=(64k)^{3s}$ and $h\leq rk$, we have $\frac{1}{64k}r^2>(s/h)^{1/s}h^2+hs/2$, a contradiction. This completes the proof.
\end{proof}

\subsection{The cube graph}

We finish the section by showing Proposition~\ref{prop:cube}, which states that $\ex^*(n, Q_8, s)\leq O_s(n^{8/5})$, where $Q_8$ denotes the graph of a three-dimensional cube. Before we start the proof, we show that almost all paths of a given length in a $K_{s, s}$-free graph are induced. Although we need this statement only for paths of length three, we prove it in full generality since we believe it might be of independent interest.

\begin{lemma}\label{lemma:many induced paths}
For any integers $k, s\geq 2$ and any $K, \eps>0$, there exists a constant $C=C(k, s, K, \eps)>0$ such that the following statement is true. Let $G$ be a graph on $n$ vertices which does not contain $K_{s, s}$ and let $V(G)=A\cup B$ be a partition of the graph $G$ with at least $e(A, B)\geq Cn^{1+\frac{1}{k}}$ crossing edges. Furthermore, assume that the graph of crossing edges is $K$-almost regular. Then, at least a $(1-\eps)$-fraction of alternating paths of length $k$ in $G$ are induced. 
\end{lemma}
\begin{proof}
We denote the graph of edges crossing the partition by $G_1$. We show that if at least an $\eps$-fraction of alternating paths of length $k$ in $G$ are not induced, then $G$ has $\Omega(nd(G_1)^{\ell-1})$ cycles of length $\ell$, for some $\ell\leq k$, with all but at most one edge crossing the partition. This suffices to show that $G$ contains $K_{s, s}$ as a subgraph.

Let us begin by showing that the total number of alternating $k$-paths in $G$ is at least $\frac{1}{2}n (\delta(G_1)/2)^k$. Namely, to specify an alternating path of length $k$, one has $n$ choices for the first vertex and at least $\delta(G_1)-k\geq \delta(G_1)/2$ choices for each of the subsequent vertices. However, since we may count every path twice, depending on the direction it is traversed, we divide by $2$. 

If an $\eps$-fraction of these paths are not induced, we have $\eps n (\delta(G_1)/2)^k/2 \geq \eps\frac{nd^k}{K^k2^{k+1}}=\Omega_{k, K, \eps}(n d^k)$ non-induced alternating paths of length $k$, where $d=d(G_1)$ is the average degree of $G_1$. In other words, there are at least $\Omega_{k, K, \eps}(n d^k)$ $k$-paths with a chord between some two vertices. By the pigeonhole principle, there exist indices $i, j\in \{0, \dots, k\}$ such that $i\leq j-2$ and at least $\frac{1}{k^2}\Omega_{k, K, \eps}(n d^k)$ of these paths have a chord between the vertices $v_i$ and $v_j$. Let us denote this collection of paths by $\cP_{ij}$.

Every path in $\cP_{ij}$ can be completed to a cycle of length $\ell=j-i+1$. Let us denote by $\cC_\ell$ the collection of $\ell$-cycles in $G$ with all but at most one edge crossing the partition. We argue that each cycle $C\in \cC_\ell$ can be obtained from at most $\ell\Delta(G_1)^{k-(j-i)}$ paths $P\in \cP_{ij}$. Given a cycle $C$, there are $\ell$ ways to label its vertices using the labels $v_i, \dots, v_j$. Furthermore, there are at most $\Delta^{i}$ ways to choose vertices $v_{i-1}, v_{i-2}, \dots, v_0$ with the restrictions $v_iv_{i-1}\in E(G_1), v_{i-1}v_{i-2}\in E(G_1)$ etc. Similarly, there are at most $\Delta^{k-j}$ ways to choose the vertices $v_{j+1}, \dots, v_k$ and therefore at most $\ell \Delta^{k-(j-i)}$ paths $P\in \cP_{ij}$ contain the cycle $C$. We conclude that $|\cC_\ell|\geq \frac{|\cP_{ij}|}{\ell\Delta^{k-\ell+1}}=\Omega_{k, K, \eps}(nd^{\ell-1})$.

The last step of the proof is to show that $|\cC_\ell|\geq \Omega_{k, K, \eps}(nd^{\ell-1})$ implies that $G$ contains $K_{s, s}$ as a subgraph. Let $\cP_{\ell-3}$ denote the collection of alternating paths of length $\ell-3$. Since at most one edge of any cycle $C\in \cC_\ell$ does not cross the partition, $C$ contains an alternating path of length $\ell-1$. Let us take this alternating path and eliminate its first and last edge to obtain an alternating path of length $\ell-3$, which we assign to the cycle $C$. Since there are at most $n\Delta^{\ell-3}$ alternating paths of length $\ell-3$, the pigeonhole principle implies that there exists a path $P\in \cP_{\ell-3}$ which is assigned to at least $\frac{|\cC_\ell|}{n\Delta^{\ell-3}}=\Omega_{k, K, \eps}(d^2)$ cycles of $\cC_\ell$.

Let us denote the endpoints of this path by $u$ and $v$. The number of ways to complete the path $P$ to a cycle in $\cC_\ell$ is upper bounded by the number of edges between $N_{G_1}(u)$ and $N_{G_1}(v)$. Both of these sets have size at most $\Delta\leq Kd$ and the number of edges in $N_{G_1}(u)\cup N_{G_1}(v)$ is at least $\Omega_{k, K, \eps}(d^2)$. On the other hand, if $G$ is $K_{s, s}$-free, the K\H{o}v\'ari-S\'os-Tur\'an theorem  implies that $N_{G_1}(u)\cup N_{G_1}(v)$ induces at most $O_s((2Kd)^{2-1/s})$ edges, which is not possible if $d$ is large enough. Hence, we conclude that $G$ contains $K_{s, s}$ as a subgraph.
\end{proof} 

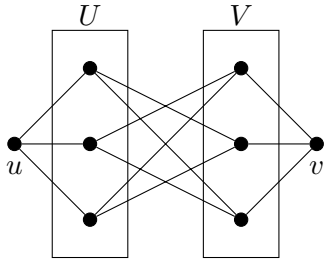
\begin{figure}[ht]
    \begin{center}
	\begin{tikzpicture}
            \node[vertex,label=below:$u$] (u) at (-2,0) {};
            \node[vertex,label=below:$v$] (v) at (2,0) {};
            \node[vertex] (a) at (-1,1) {};
            \node[vertex] (b) at (-1,0) {};
            \node[vertex] (c) at (-1,-1) {};
            \node[vertex] (x) at (1,1) {};
            \node[vertex] (y) at (1,0) {};
            \node[vertex] (z) at (1,-1) {};

            \draw (u) edge (a);
            \draw (u) edge (b);
            \draw (u) edge (c);
            \draw (v) edge (x);
            \draw (v) edge (y);
            \draw (v) edge (z);

            \draw (a) -- (y) -- (c) -- (x) -- (b) -- (z) -- (a);

            \draw (-1.5, -1.5) rectangle (-0.5, 1.5) ; \node at (-1, 1.7) {$U$};
            \draw (0.5, -1.5) rectangle (1.5, 1.5) ; \node at (1, 1.7) {$V$};
        \end{tikzpicture}
    \end{center}
    \caption{Illustration of the proof of Proposition~\ref{prop:cube}.}
    \label{fig1}
\end{figure}

We are now ready to prove the upper bound on the induced Tur\'an number of the cube.

\begin{proof}[Proof of Proposition~\ref{prop:cube}.]
Our goal is to  find a pair of non-adjacent vertices $u, v$ with a large number of induced paths of length three between them. Then, we find an alternating induced cycle of length $6$ between the neighbourhoods of $u$ and $v$, which suffices to find the graph of the cube as depicted in Figure~\ref{fig1}.

Let $G$ be an $n$ vertex graph with average degree at least $Cn^{3/5}$ containing no $K_{s, s}$. Partition the vertex set of $G$ into parts $A, B$ such that at least half of the edges cross the partition, and let $G_1$ be the graph of crossing edges. By Lemma~\ref{lemma:passing to an almost regular subgraph} applied with $\alpha=\frac{3}{5}$, $G_1$ has a $K$-almost-regular induced subgraph $H_1$ with $m$ vertices and average degree $d$ satisfying $d\geq C_0m^{3/5}$, where $C_0=\frac{C}{4}$ and $K=2^{10}$. Let $H$ be the subgraph of $G$ induced on $V(H_1)$.

By Lemma~\ref{lemma:many induced paths}, the graph $H$ contains $\Omega(md^3)$ alternating induced paths of length $3$. By the pigeonhole principle, there exists a pair of vertices $u, v$, which have at least $\Omega(md^3/m^2)=\Omega(d^3/m)$ induced alternating paths of length $3$ between them. In particular, the vertices $u$ and $v$ are non-adjacent.

Let us denote the collection of these paths by $\cP$, the set of all neighbours of $u$ on these paths by $U$ and the set of neighbours of $v$ on these paths by $V$. Since all paths are induced, $U\cap N(v)=\emptyset$ and $V\cap N(u)=\emptyset$. Moreover, the number of induced alternating paths of length $3$ between $u$ and $v$ is equal to the number of edges between $U$ and $V$, meaning that $e(U, V)=\Omega(d^3/m)$. On the other hand, the cardinality of $U$ and $V$ is at most $\Delta(H_1)\leq Kd$. 

We claim that if $C_0$ is sufficiently large, one has $e(U, V)\geq C_0(|U|+|V|)^{4/3}$, which allows us to apply Theorem~\ref{thm:alternating cycles} with $k=3$. To verify that there are sufficiently many edges between $U$ and $V$, we recall that $d\geq C_0 m^{3/5}$, which implies
\[e(U, V)=\Omega(d^3/m)\geq \Omega\left(d^{4/3} C_0^{5/3}\right)\geq C_0\cdot (2Kd)^{4/3}\geq C_0(|U|+|V|)^{4/3}.\]
Thus, one can find an alternating induced cycle of length $6$ between $U$ and $V$. But the vertices $u$ and $v$, together with this cycle, form an induced copy of $Q_8$, which completes the proof.
\end{proof}

\section{Concluding remarks}

In this paper, we proposed a framework which unifies the study of Tur\'an-type problems with the study of induced subgraphs. We proved upper bounds on $\ex^*(n, H, s)$ with the same asymptotic behavior as the best known upper bounds on the usual Tur\'an number $\ex(n, H)$ for several natural classes of bipartite graphs, e.g. when $H$ is a tree, cycle or has degrees on one side bounded by $k$. Let us repeat our conjecture that a similar result should hold for every bipartite graph $H$.

\mainconjecture*

An obvious difficulty in the resolution of this conjecture is that we do not even know the extremal numbers of most bipartite graphs. However, it is plausible that this can be circumvented by some clever argument. It would be also interesting to prove or disprove this conjecture in case $H$ is replaced with a family of bipartite graphs $\mathcal{H}$. As we discussed in the introduction, sharp bounds are known in case $\mathcal{H}$ is the family of subdivisions of a given graph $H$.

A curious problem left open is about the family of graphs of VC-dimension at most $d$, which is of great interest due to its connection to geometry. This family can be defined  as the family of graphs containing no induced member of the following finite collection of graphs $\cH$. The collection $\cH$ contains all graphs $H$ with a partition $A\cup B$ such that $|A|=d+1$, $|B|=2^{d+1}$, and for every $X\subset A$ there is a unique $b\in B$ such that $b$ is connected to all vertices in $X$, but no vertices in $A\setminus X$. For $d\geq 3$, the best known upper bound is $\ex^{*}(n,\cH,s)=o(n^{2-1/d})$ due to Janzer and Pohoata \cite{JP}. However, we believe that $\ex^{*}(n,\cH,s)=O(n^{2-1/d-\delta})$ should also hold for some $\delta=\delta(d)>0$. It seems the main obstacle in proving this is that the best known bound for the ordinary Tur\'an number is also $\ex(n,\cH)=o(n^{2-1/d})$, following from a result of Sudakov and Tomon \cite{ST20} as well.

\medskip
\noindent
\textbf{Acknowledgments.} We thank anonymous referees for numerous thoughtful comments and the observation that the original statement of the Conjecture~\ref{conj:mainconjecture} had a trivial counterexample, if $H$ is allowed to be disconnected. We note that a similar observation was made by Sheffield \cite{S24} too.

\end{document}